\documentclass[a4paper,english,fontsize=10pt,parskip=half,abstracton]{scrartcl}
\usepackage{babel}
\usepackage[utf8]{inputenc}
\usepackage[T1]{fontenc}
\usepackage[a4paper,left=20mm,right=20mm,top=30mm,bottom=30mm]{geometry}
\usepackage{amsmath}
\usepackage{amsthm}
\usepackage{amssymb}
\usepackage{enumerate}
\usepackage[bookmarks=false,
            pdftitle={2-Blocks with minimal nonabelian defect groups},
            pdfauthor={Benjamin Sambale},
            pdfkeywords={2-blocks, defect groups, minimal nonabelian, Redei, fusion systems},
            pdfstartview={FitH}]{hyperref}
\usepackage[arrow, matrix]{xy}
\newtheorem{Theorem}{Theorem}[section]
\newtheorem{Lemma}[Theorem]{Lemma}
\newtheorem{Proposition}[Theorem]{Proposition}

\renewcommand{\phi}{\varphi}
\newcommand{\C}{\operatorname{C}}
\newcommand{\N}{\operatorname{N}}
\newcommand{\Z}{\operatorname{Z}}
\newcommand{\cohom}{\operatorname{H}}
\newcommand{\Aut}{\operatorname{Aut}}
\newcommand{\Inn}{\operatorname{Inn}}
\newcommand{\Out}{\operatorname{Out}}
\newcommand{\pcore}{\operatorname{O}}
\newcommand{\GL}{\operatorname{GL}}

\newcommand{\Irr}{\operatorname{Irr}}
\newcommand{\IBr}{\operatorname{IBr}}
\newcommand{\E}{\operatorname{E}}
\newcommand{\F}{\operatorname{F}}
\newcommand{\Bl}{\operatorname{Bl}}
\newcommand{\inertial}{\operatorname{T}}
\newcommand{\Gal}{\operatorname{Gal}}
\newcommand{\Ind}{\operatorname{Ind}}
\newcommand{\Ob}{\operatorname{Ob}}
\newcommand{\Hom}{\operatorname{Hom}}

\mathchardef\ordinarycolon\mathcode`\:  
 \mathcode`\:=\string"8000
 \begingroup \catcode`\:=\active
   \gdef:{\mathrel{\mathop\ordinarycolon}}
 \endgroup

\title{$2$-Blocks with minimal nonabelian defect groups}
\author{
Benjamin Sambale\\
Mathematisches Institut\\
Friedrich-Schiller-Universität\\
07743 Jena\\
Germany\\
{\tt benjamin.sambale@uni-jena.de}
}
\date{\today}

\begin{document}
\frenchspacing
\maketitle
\begin{abstract}\noindent
We study numerical invariants of $2$-blocks with minimal nonabelian defect groups. These groups were classified by Rédei (see \cite{Redei}). If the defect group is also metacyclic, then the block invariants are known (see \cite{Sambale}). In the remaining cases there are only two (infinite) families of “interesting” defect groups. In all other cases the blocks are nilpotent. We prove Brauer's $k(B)$-conjecture and the Olsson-conjecture for all $2$-blocks with minimal nonabelian defect groups. For one of the two families we also show that Alperin's weight conjecture and Dade's conjecture is satisfied. This paper is a part of the author's PhD thesis.
\end{abstract}
\textbf{Keywords:} blocks of finite groups, minimal nonabelian defect groups, Alperin's conjecture, Dade's conjecture.
\tableofcontents
\section{Introduction}
Let $R$ be a discrete complete valuation ring with quotient field $K$ of characteristic $0$. Moreover, let $(\pi)$ be the maximal ideal of $R$ and $F:=R/(\pi)$. We assume that $F$ is algebraically closed of characteristic $2$.
We fix a finite group $G$, and assume that $K$ contains all $|G|$-th roots of unity. Let $B$ be a block of $RG$ with defect group $D$. We denote the number of irreducible ordinary characters of $B$ by $k(B)$. These characters split in $k_i(B)$ characters of height $i\in\mathbb{N}_0$. Similarly, let $k^i(B)$ be the number of characters of defect $i\in\mathbb{N}_0$. Finally, let $l(B)$ be the number of irreducible Brauer characters of $B$.
The defect group $D$ is called \emph{minimal nonabelian} if every proper subgroup of $D$ is abelian, but not $D$ itself. Rédei has shown that $D$ is isomorphic to one of the following groups (see \cite{Redei}):
\begin{enumerate}[(i)]
\item $\langle x,y\mid x^{2^r}=y^{2^s}=1,\ xyx^{-1}=y^{1+2^{s-1}}\rangle$, where $r\ge 1$ and $s\ge 2$,
\item $\langle x,y\mid x^{2^r}=y^{2^s}=[x,y]^2=[x,x,y]=[y,x,y]=1\rangle$, where $r\ge s\ge 1$, $[x,y]:=xyx^{-1}y^{-1}$ and $[x,x,y]:=[x,[x,y]]$,\label{redei2}
\item $Q_8$.
\end{enumerate}
In the first and last case $D$ is also metacyclic. In this case $B$ is well understood (see \cite{Sambale}). Thus, we may assume that $D$ has the form \eqref{redei2}. 

\section{Fusion systems}
To analyse the possible fusion systems on $D$ we start with a group theoretical lemma.

\begin{Lemma}\label{lemmamna}
Let $z:=[x,y]$. Then the following hold:
\begin{enumerate}[(i)]
\item $|D|=2^{r+s+1}$.\label{lemmamna1}
\item $\Phi(D)=\Z(D)=\langle x^2,y^2,z\rangle\cong C_{2^{r-1}}\times C_{2^{s-1}}\times C_2$.\label{lemmamna2}
\item $D'=\langle z\rangle\cong C_2$.
\item $|\Irr(D)|=5\cdot 2^{r+s-2}$.
\item\label{lemmamna4} If $r=s=1$, then $D\cong D_8$. For $r\ge 2$ the maximal subgroups of $D$ are given by 
\begin{align*}
\langle x^2,y,z\rangle&\cong C_{2^{r-1}}\times C_{2^s}\times C_2,\\
\langle x,y^2,z\rangle&\cong C_{2^r}\times C_{2^{s-1}}\times C_2,\\
\langle xy,x^2,z\rangle&\cong C_{2^r}\times C_{2^{s-1}}\times C_2.
\end{align*}
\end{enumerate}
\end{Lemma}

We omit the (elementary) proof of this lemma. However, notice that $|P'|=2$ and $|P:\Phi(P)|=|P:\Z(P)|=p^2$ hold for every minimal nonabelian $p$-group $P$.
Rédei has also shown that for different pairs $(r,s)$ one gets nonisomorphic groups. This gives precisely $\bigl[\frac{n-1}{2}\bigr]$ isomorphism classes of these groups of order $2^n$.
For $r\ne 1$ (that is $|D|\ge 16$) the structure of the maximal subgroups shows that all these groups are nonmetacyclic.

Now we investigate the automorphism groups.

\begin{Lemma}\label{mnaaut}
The automorphism group $\Aut(D)$ is a $2$-group, if and only if $r\ne s$ or $r=s=1$.
\end{Lemma}
\begin{proof}
If $r\ne s$ or $r=s=1$, then there exists a characteristic maximal subgroup of $D$ by Lemma~\ref{lemmamna}\eqref{lemmamna4}.
In these cases $\Aut(D)$ must be a $2$-group. Thus, we may assume $r=s\ge 2$. Then one can show that the map $x\mapsto y$, $y\mapsto x^{-1}y^{-1}$ is an automorphism of order $3$.
\end{proof}


\begin{Lemma}\label{homocaut}
Let $P\cong C_{2^{n_1}}\times\ldots\times C_{2^{n_k}}$ with $n_1,\ldots,n_k,k\in\mathbb{N}$. Then $\Aut(P)$ is a $2$-group, if and only if the $n_i$ are pairwise distinct.
\end{Lemma}
\begin{proof}
See for example Lemma~2.7 in \cite{OlssonRedei}.
\end{proof}

Now we are able to decide, when a fusion system on $D$ is nilpotent. 

\begin{Theorem}\label{mnafusion}
Let $\mathcal{F}$ be a fusion system on $D$. Then $\mathcal{F}$ is nilpotent or $s=1$ or $r=s$. If $r=s\ge 2$, then $\mathcal{F}$ is controlled by $D$.
\end{Theorem}
\begin{proof}
We assume $s\ne 1$. Let $Q<D$ be an $\mathcal{F}$-essential subgroup. Since $Q$ is also $\mathcal{F}$-centric, we get $\C_P(Q)=Q$. This shows that $Q$ is a maximal subgroup of $D$. By Lemma~\ref{lemmamna}\eqref{lemmamna4} and Lemma~\ref{homocaut}, one of the following holds:
\begin{enumerate}[(i)]
\item $r=2\ (\,=s)$ and $Q\in\{\langle x^2,y,z\rangle,\langle x,y^2,z\rangle,\langle xy,x^2,z\rangle\}$,
\item $r>s=2$ and $Q\in\bigl\{\langle x,y^2,z\rangle,\langle xy,x^2,z\rangle\bigr\}$,
\item $r=s+1$ and $Q=\langle x^2,y,z\rangle$.
\end{enumerate}
In all cases $\Omega(Q)\subseteq\Z(P)$. Let us consider the action of $\Aut_{\mathcal{F}}(Q)$ on $\Omega(Q)$. The subgroup $1\ne P/Q=\N_P(Q)/\C_P(Q)\cong\Aut_P(Q)\le\Aut_\mathcal{F}(Q)$ acts trivially on $\Omega(Q)$. On the other hand every nontrivial automorphism of odd order acts nontrivially on $\Omega(Q)$ (see for example 8.4.3 in \cite{Kurzweil}). Hence, the kernel of this action is a nontrivial normal $2$-subgroup of $\Aut_\mathcal{F}(Q)$. In particular $\pcore_2(\Aut_\mathcal{F}(Q))\ne 1$. 
But then $\Aut_\mathcal{F}(Q)$ cannot contain a strongly $2$-embedded subgroup. 

This shows that there are no $\mathcal{F}$-essential subgroups. Now the claim follows from Lemma~\ref{mnaaut} and Alperin's fusion theorem.
\end{proof}

Now we consider a kind of converse. If $r=s=1$, then there are nonnilpotent fusion systems on $D$. In the case $r=s\ge 2$ one can construct a nonnilpotent fusion system with a suitable semidirect product (see Lemma~\ref{mnaaut}). We show that there is also a nonnilpotent fusion system in the case $r>s=1$.

\begin{Proposition}\label{nnilfus1}
If $s=1$, then there exists a nonnilpotent fusion system on $D$. 
\end{Proposition}
\begin{proof}
We may assume $r\ge 2$. Let $A_4$ be the alternating group of degree $4$, and let $H:=\langle\widetilde{x}\rangle\cong C_{2^r}$.
Moreover, let $\phi:H\to\Aut(A_4)\cong S_4$ such that $\phi_{\widetilde{x}}\in\Aut(A_4)$ has order $4$.
Write $\widetilde{y}:=(12)(34)\in A_4$ and choose $\phi$ such that $\phi_{\widetilde{x}}(\widetilde{y}):=(13)(24)$. Finally, let $G:=A_4\rtimes_\phi H$. Since all $4$-cycles in $S_4$ are conjugate, $G$ is uniquely determined up to isomorphism.
Because $[\widetilde{x},\widetilde{y}]=(13)(24)(12)(34)=(14)(23)$, we get $\langle\widetilde{x},\widetilde{y}\rangle\cong D$. 
The fusion system $\mathcal{F}_G(D)$ is nonnilpotent, since $A_4$ (and therefore $G$) is not $2$-nilpotent.
\end{proof}

\section{The case $r>s=1$}\label{secrs1}
Now we concentrate on the case $r>s=1$, i.e.
\[D:=\langle x,y\mid x^{2^r}=y^2=[x,y]^2=[x,x,y]=[y,x,y]=1\rangle\]
with $r\ge 2$. As before $z:=[x,y]$. We also assume that $B$ is a nonnilpotent block. By Lemma~\ref{mnaaut}, $\Aut(D)$ is a $2$-group, and the inertial index $t(B)$ of $B$ equals $1$.

\subsection{The $B$-subsections}
Olsson has already obtained the conjugacy classes of so called $B$-subsections (see \cite{OlssonRedei}). However, his results contain errors. For example he missed the necessary relations $[x,x,y]$ and $[y,x,y]$ in the definition of $D$. 

In the next lemma we denote by $\Bl(RH)$ the set of blocks of a finite group $H$. If $H\le G$ and $b\in\Bl(RH)$, then $b^G$ is the Brauer correspondent of $b$ (if exists). Moreover, we use the notion of subpairs and subsections (see \cite{Olssonsubpairs}).

\begin{Lemma}\label{repkonjs1}
Let $b\in\Bl(RD\C_G(D))$ be a Brauer correspondent of $B$. For $Q\le D$ let $b_Q\in\Bl(RQ\C_G(Q))$ such that $(Q,b_Q)\le(D,b)$. Set $\mathcal{T}:=\Z(D)\cup\{x^iy^j:i,j\in\mathbb{Z},\ i\text{ odd}\}$. Then
\[\bigcup_{a\in\mathcal{T}}{\Bigl\{\bigl(a,b_{\C_D(a)}^{\C_G(a)}\bigr)\Bigr\}}\]
is a system of representatives for the conjugacy classes of $B$-subsections. Moreover, $|\mathcal{T}|=2^{r+1}$.
\end{Lemma}
\begin{proof}
If $r=2$, then the claim follows from Proposition~2.14 in \cite{OlssonRedei}. For $r\ge 3$ the same argument works. However, Olsson refers wrongly to Proposition~2.11 (the origin of this mistake already lies in Lemma~2.8). 
\end{proof}

From now on we write $b_a:=b_{\C_D(a)}^{\C_G(a)}$ for $a\in\mathcal{T}$.

\begin{Lemma}\label{auts1}
Let $P\cong C_{2^s}\times C_2^2$ with $s\in\mathbb{N}$, and let $\alpha$ be an automorphism of $P$ of order $3$. Then $\C_P(\alpha):=\{b\in P:\alpha(b)=b\}\cong C_{2^s}$. 
\end{Lemma}
\begin{proof}
We write $P=\langle a\rangle\times\langle b\rangle\times\langle c\rangle$ with $|\langle a\rangle|=2^s$. It is well known that the kernel of the restriction map $\Aut(P)\to\Aut(P/\Phi(P))$ is a $2$-group. Since $|\Aut(P/\Phi(P))|=|\GL(3,2)|=168=2^3\cdot 3\cdot 7$, it follows that $|\Aut(P)|$ is divisible by $3$ only once. In particular every automorphism of $P$ of order $3$ is conjugate to $\alpha$ or $\alpha^{-1}$. Thus, we may assume $\alpha(a)=a$, $\alpha(b)=c$ and $\alpha(c)=bc$. Then $\C_P(\alpha)=\langle a\rangle\cong C_{2^s}$.
\end{proof}

\subsection{The numbers $k(B)$, $k_i(B)$ and $l(B)$}
The next step is to determine the numbers $l(b_a)$. The case $r=2$ needs special attention, because in this case $D$ contains an elementary abelian maximal subgroup of order $8$. We denote the inertial group of a block $b\in\Bl(RH)$ with $H\unlhd G$ by $\inertial_G(b)$.

\begin{Lemma}\label{lokallb}
There is an element $c\in\Z(D)$ of order $2^{r-1}$ such that $l(b_a)=1$ for all $a\in\mathcal{T}\setminus\langle c\rangle$.
\end{Lemma}
\begin{proof}\hfill\\
\textbf{Case 1:} $a\in\Z(D)$. \\
Then $b_a=b_D^{\C_G(a)}$ is a block with defect group $D$ and Brauer correspondent $b_D\in\Bl(RD\C_{\C_G(a)}(D))$. 
Let $M:=\langle x^2,y,z\rangle\cong C_{2^{r-1}}\times C_2^2$. Since $B$ is nonnilpotent, there exists an element $\alpha\in\inertial_{\N_G(M)}(b_M)$ such that $\alpha\C_G(M)\in\inertial_{\N_G(M)}(b_M)/\C_G(M)$ has order $q\in\{3,7\}$. 
We will exclude the case $q=7$. In this case $r=2$ and $\inertial_{\N_G(M)}(b_M)/\C_G(M)$ is isomorphic to a subgroup of $\Aut(M)\cong\GL(3,2)$.
Since
\[(M,{^d b_M})={^d(M,b_M)}\le{^d(D,b_D)}=(D,b_D)\]
for all $d\in D$, we have $D\subseteq\inertial_{\N_G(M)}(b_M)$. This implies $\inertial_{\N_G(M)}(b_M)/\C_G(M)\cong\GL(3,2)$, because $\GL(3,2)$ is simple. By Satz~1 in \cite{Bender}, this contradicts the fact that $\inertial_{\N_G(M)}(b_M)/\C_G(M)$ contains a strongly $2$-embedded subgroup (of course this can be shown “by hand” without invoking \cite{Bender}). Thus, we have shown $q=3$. Now
\[\inertial_{\N_G(M)}(b_M)/\C_G(M)\cong S_3\]
follows easily.
By Lemma~\ref{auts1} there is an element $c:=x^{2i}y^jz^k\in\C_M(\alpha)$ ($i,j,k\in\mathbb{Z}$) of order $2^{r-1}$. Let us assume that $j$ is odd. Since $x\alpha x\equiv x\alpha x^{-1}\equiv\alpha^{-1}\pmod{C_G(M)}$ we get
\begin{align*}
\alpha(x^{2i}y^jz^{k+1})\alpha^{-1}&=\alpha x(x^{2i}y^jz^k)x^{-1}\alpha^{-1}=x\alpha^{-1}(x^{2i}y^jz^k)\alpha x^{-1}\\
&=x(x^{2i}y^jz^k)x^{-1}=x^{2i}y^jz^{k+1}.
\end{align*}
But this contradicts Lemma~\ref{auts1}. Hence, we have proved that $j$ is even. In particular $c\in\Z(D)$. For $a\notin\langle c\rangle$ we have $\alpha\notin\C_G(a)$ and $l(b_a)=1$. While in the case $a\in\langle c\rangle$ we get $\alpha\in\C_G(a)$, and $b_a$ is nonnilpotent. Thus, in this case $l(b_a)$ remains unknown.

\textbf{Case 2:} $a\notin\Z(D)$. \\
Let $\C_D(a)=\langle\Z(D),a\rangle=:M$. Since $(M,b_M)$ is a Brauer subpair, $b_M$ has defect group $M$. It follows from $(M,b_M)\unlhd (D,b_D)$ that also $b_a$ has defect group $M$ and Brauer correspondent $b_M$. In case $M\cong C_{2^r}\times C_2$ we get $l(b_a)=1$. Now let us assume $M\cong C_{2^{r-1}}\times C_2^2$. 
As in the first case, we choose $\alpha\in\inertial_{\N_G(M)}(b_M)$ such that $\alpha\C_G(M)\in\inertial_{\N_G(M)}(b_M)/\C_G(M)$ has order $3$. Since $a\notin\Z(D)$, we derive $\alpha\notin\C_G(a)$ and $t(b_a)=l(b_a)=1$. 
\end{proof}

We denote by $\IBr(b_u):=\{\phi_u\}$ for $u\in\mathcal{T}\setminus\langle c\rangle$ the irreducible Brauer character of $b_u$. Then the generalized decomposition numbers $d^u_{\chi\phi_u}$ for $\chi\in\Irr(B)$ form a column $d(u)$. Let $2^k$ be the order of $u$, and let $\zeta:=\zeta_{2^k}$ be a primitive $2^k$-th root of unity. Then the entries of $d(u)$ lie in the ring of integers $\mathbb{Z}[\zeta]$. Hence, there exist integers $a_i^{u}(\chi)\in\mathbb{Z}$ such that
\[d_{\chi\phi_u}^u=\sum_{i=0}^{2^{k-1}-1}{a_i^{u}(\chi)\zeta^i}.\]
We expand this by
\[a_{i+2^{k-1}}^u:=-a_i^u\]
for all $i\in\mathbb{Z}$.

Let $|G|=2^am$ where $2\nmid m$. We may assume $\mathbb{Q}(\zeta_{|G|})\subseteq K$. Then $\mathbb{Q}(\zeta_{|G|})|\mathbb{Q}(\zeta_m)$ is a Galois extension, and we denote the corresponding Galois group by
\[\mathcal{G}:=\Gal(\mathbb{Q}(\zeta_{|G|})|\mathbb{Q}(\zeta_m)).\] 
Restriction gives an isomorphism
\[\mathcal{G}\cong\Gal(\mathbb{Q}(\zeta_{2^a})|\mathbb{Q}).\]
In particular $|\mathcal{G}|=2^{a-1}$.
For every $\gamma\in\mathcal{G}$ there is a number $\widetilde{\gamma}\in\mathbb{N}$ such that $\gcd(\widetilde{\gamma},|G|)=1$, $\widetilde{\gamma}\equiv 1\pmod{m}$, and $\gamma(\zeta_{|G|})=\zeta_{|G|}^{\widetilde{\gamma}}$ hold. Then $\mathcal{G}$ acts on the set of subsections by
\[^\gamma(u,b):=(u^{\widetilde{\gamma}},b).\]
For every $\gamma\in\mathcal{G}$ we get
\[d(u^{\widetilde{\gamma}})=\sum_{s\in\mathcal{S}}{a_s^{u}\zeta_{2^k}^{s\widetilde{\gamma}}}\]
for every system $\mathcal{S}$ of representatives of the cosets of $2^{k-1}\mathbb{Z}$ in $\mathbb{Z}$.
It follows that
\begin{equation}\label{aiuspur}
a_s^u=2^{1-a}\sum_{\gamma\in\mathcal{G}}{d\bigl(u^{\widetilde{\gamma}}\bigr)\zeta_{2^k}^{-\widetilde{\gamma}s}}
\end{equation}
for $s\in\mathcal{S}$.

Now let $u\in\mathcal{T}\setminus\Z(D)$ and $M:=\C_D(u)$. 
Then $b_u$ and $b_M^{\inertial_{\N_G(M)}(b_M)\cap \N_G(\langle u\rangle)}$ have $M$ as defect group, because $D\nsubseteq \N_G(\langle u\rangle)$. By (6B) in \cite{BrauerBlSec2} it follows that the $2^{r-1}$ distinct $B$-subsections of the form $^\gamma(u,b_u)$ with $\gamma\in\mathcal{G}$ are pairwise nonconjugate. The same holds for $u\in\Z(D)\setminus\{1\}$. Using this and equation~\eqref{aiuspur} we can adapt Lemma~3.9 in \cite{Olsson}:

\begin{Lemma}\label{orthogonal}
Let $c\in\Z(D)$ as in Lemma~\ref{lokallb}, and let $u,v\in\mathcal{T}\setminus\langle c\rangle$ with $|\langle u\rangle|=2^k$ and $|\langle v\rangle|=2^l$. Moreover, let $i\in\{0,1,\ldots,2^{k-1}-1\}$ and $j\in\{0,1,\ldots,2^{l-1}-1\}$. 
If there exist $\gamma\in\mathcal{G}$ and $g\in G$ such that $^g(u,b_u)={^\gamma(v,b_v)}$, then
\[(a_i^u,a_j^v)=
\begin{cases} 2^{d(B)-k+1}&\text{if }u\in\Z(D)\text{ and }j\widetilde{\gamma}-i\equiv 0\pmod{2^k}\\
-2^{d(B)-k+1}&\text{if }u\in\Z(D)\text{ and }j\widetilde{\gamma}-i\equiv 2^{k-1}\pmod{2^k}\\
2^{d(B)-k}&\text{if }u\notin\Z(D)\text{ and }j\widetilde{\gamma}-i\equiv 0\pmod{2^k}\\
-2^{d(B)-k}&\text{if }u\notin\Z(D)\text{ and }j\widetilde{\gamma}-i\equiv 2^{k-1}\pmod{2^k}\\
0&\text{otherwise}\end{cases}.\]
Otherwise $(a_i^u,a_j^v)=0$. In particular $(a_i^u,a_j^v)=0$ if $k\ne l$.
\end{Lemma}

Using the theory of contributions we can also carry over Lemma~(6.E) in \cite{Kuelshammerwr}:

\begin{Lemma}\label{neukirchapp222}
Let $u\in\Z(D)$ with $l(b_u)=1$. If $u$ has order $2^k$, then for every $\chi\in\Irr(B)$ holds:
\begin{enumerate}[(i)]
\item $2^{h(\chi)}\mid a_i^u(\chi)$ for $i=0,\ldots,2^{k-1}-1$,
\item $\sum\limits_{i=0}^{2^{k-1}-1}{a_i^u(\chi)}\equiv 2^{h(\chi)}\pmod{2^{h(\chi)+1}}$.
\end{enumerate}
\end{Lemma}

By Lemma~1.1 in \cite{Robinson} we have
\begin{equation}\label{robinsonkbinequ}
k(B)\le\sum_{i=0}^{\infty}{2^{2i}k_i(B)}\le|D|.
\end{equation}
In particular Brauer's $k(B)$-conjecture holds. Olsson's conjecture 
\begin{equation}\label{k0inequ}
k_0(B)\le|D:D'|=2^{r+1}
\end{equation}
follows by Theorem~3.1 in \cite{Robinson}. Now we are able to calculate the numbers $k(B)$, $k_i(B)$ and $l(B)$.

\begin{Theorem}\label{s1klki}
We have
\[k(B)=5\cdot 2^{r-1}=|\Irr(D)|,\hspace{5mm}k_0(B)=2^{r+1}=|D:D'|,\hspace{5mm}k_1(B)=2^{r-1},\hspace{5mm}l(B)=2.\]
\end{Theorem}
\begin{proof}
We argue by induction on $r$. Let $r=2$, and let $c\in\Z(D)$ as in Lemma~\ref{lokallb}. By way of contradiction we assume $c=z$. If $\alpha$ and $M$ are defined as in the proof of Lemma~\ref{lokallb}, then $\alpha$ acts nontrivially on $M/\langle z\rangle\cong C_2^2$. On the other hand $x$ acts trivially on $M/\langle z\rangle$. This contradicts $x\alpha x^{-1}\alpha\in\C_G(M)$.

This shows $c\in\{x^2,x^2z\}$ and $D/\langle c\rangle\cong D_8$. Thus, we can apply Theorem~2 in \cite{Brauer}. For this let
\[M_1:=\begin{cases}\langle x,z\rangle&\text{if }c=x^2\\\langle xy,z\rangle&\text{if }c=x^2z\end{cases}.\]
Then $M\ne M_1\cong C_4\times C_2$ and $\overline{M}:=M/\langle c\rangle\cong C_2^2\cong M_1/\langle c\rangle=:\overline{M_1}$. Let $\beta$ be the block of $R\overline{\C_G(c)}:=R[\C_G(c)/\langle c\rangle]$ which is dominated by $b_c$.
By Theorem~1.5 in \cite{Olsson} we have 
\[3\ \mbox{\Large{$\mid$}}\ |\inertial_{\N_{\overline{\C_G(c)}}(\overline{M})}(\beta_{\overline{M}})/\C_{\overline{\C_G(c)}}(\overline{M})|\]
and
\[3\ \mbox{\Large{$\nmid$}}\ |\inertial_{\N_{\overline{\C_G(c)}}(\overline{M_1})}(\beta_{\overline{M_1}})/\C_{\overline{\C_G(c)}}(\overline{M_1})|,\]
where $(\overline{M},\beta_{\overline{M}})$ and $(\overline{M_1},\beta_{\overline{M_1}})$ are $\beta$-subpairs.
This shows that case $(ab)$ in Theorem~2 in \cite{Brauer} occurs.
Hence, $l(b_c)=l(\beta)=2$. Now Lemma~\ref{lokallb} yields
\[k(B)\ge 1+k(B)-l(B)=9.\]
It is well known that $k_0(B)$ is divisible by $4$. Thus, the equations~\eqref{robinsonkbinequ} and \eqref{k0inequ} imply $k_0(B)=8$. Moreover,
\[d_{\chi\phi_z}^z=a_0^z(\chi)=\pm1\]
holds for every $\chi\in\Irr(B)$ with $h(\chi)=0$.
This shows $4k_1(B)\le |D|-k_0(B)=8$. It follows that $k_1(B)=l(B)=2$.

Now we consider the case $r\ge 3$. Since $z$ is not a square in $D$, we have $z\notin\langle c\rangle$. Let $a\in\langle c\rangle$ such that $|\langle a\rangle|=2^k$. If $k=r-1$, then $l(b_a)=2$ as before. Now let $k<r-1$. Then $D/\langle a\rangle$ has the same isomorphism type as $D$, but one has to replace $r$ by $r-k$. By induction we get $l(b_a)=2$ for $k\ge 1$. This shows
\[k(B)\ge 1+k(B)-l(B)=2^{r+1}+2^{r-1}-1.\]
Equation~\eqref{robinsonkbinequ} yields
\begin{align*}
2^{r+2}-4&=2^{r+1}+4(2^{r-1}-1)\le k_0(B)+4(k(B)-k_0(B))\\[1mm]
&\le\sum_{i=0}^{\infty}{2^{2i}k_i(B)}\le|D|=2^{r+2}.
\end{align*}
Now the conclusion follows easily.
\end{proof}

As a consequence, Brauer's height zero conjecture and the Alperin-McKay-conjecture hold for $B$. 

\subsection{Generalized decomposition numbers}
Now we will determine some of the generalized decomposition numbers. 
Again let $c\in\Z(D)$ as in Lemma~\ref{lokallb}, and let $u\in\Z(D)\setminus\langle c\rangle$ with $|\langle u\rangle|=2^k$. Then $(a_i^u,a_i^u)=2^{r+3-k}$ and $2\mid a_i^u(\chi)$ for $h(\chi)=1$ and $i=0,\ldots,2^{k-1}-1$. This gives
\[|\{\chi\in\Irr(B):a_i^u(\chi)\ne 0\}|\le 2^{r+3-k}-3|\{\chi\in\Irr(B):h(\chi)=1,\ a_i^u(\chi)\ne 0\}|.\]
Moreover, for every character $\chi\in\Irr(B)$ there exists $i\in\{0,\ldots,2^{k-1}-1\}$ such that $a_i^u(\chi)\ne 0$. Hence,
\begin{align*}
k(B)&\le\sum_{i=0}^{2^{k-1}-1}{\sum_{\substack{\chi\in\Irr(B),\\a_i^u(\chi)\ne 0}}{1}}\le\sum_{i=0}^{2^{k-1}-1}{\Biggl(2^{r+3-k}-3\sum_{\substack{\chi\in\Irr(B),\\h(\chi)=1,\\a_i^u(\chi)\ne 0}}{1}\Biggr)}=|D|-3\sum_{i=0}^{2^{k-1}-1}{\sum_{\substack{\chi\in\Irr(B),\\h(\chi)=1,\\a_i^u(\chi)\ne 0}}{1}}\\
&\le |D|-3k_1(B)=k(B).
\end{align*}
This shows that for every $\chi\in\Irr(B)$ there exists $i(\chi)\in\{0,\ldots,2^{k-1}-1\}$ such that
\[d_{\chi\phi_u}^u=\begin{cases}\pm\zeta_{2^k}^{i(\chi)}&\text{if }h(\chi)=0\\\pm2\zeta_{2^k}^{i(\chi)}&\text{if }h(\chi)=1\end{cases}.\]
In particular
\[d_{\chi\phi_u}^u=a_0^u(\chi)=\begin{cases}\pm1&\text{if }h(\chi)=0\\\pm2&\text{if }h(\chi)=1\end{cases}\]
for $k=1$.

By Lemma~\ref{orthogonal} we have $(a_i^u,a_i^u)=4$ for $u\in\mathcal{T}\setminus\Z(D)$ and $i=0,\ldots,2^{r-1}-1$. If $a_i^u$ has only one nonvanishing entry, then $a_i^u$ would not be orthogonal to $a_0^z$. Hence, $a_i^u$ has up to ordering the form
\[(\pm1,\pm1,\pm1,\pm1,0,\ldots,0)^{\text{T}},\]
where the signs are independent of each other. The proof of Theorem~3.1 in \cite{Robinson} gives
\[|d_{\chi\phi_u}^u|=1\]
for $u\in\mathcal{T}\setminus\Z(D)$ and $\chi\in\Irr(B)$ with $h(\chi)=0$.
In particular $d_{\chi\phi_u}^u=0$ for characters $\chi\in\Irr(B)$ of height $1$. By suitable ordering we get
\[a_i^u(\chi_j)=\begin{cases}\pm1&\text{if }j-4i\in\{1,\ldots,4\}\\0&\text{otherwise}\end{cases}\text{ and }d_{\chi_j\phi_u}^u=\begin{cases}\pm\zeta_{2^r}^{[\frac{j-1}{4}]}&\text{if }1\le j\le k_0(B)\\0&\text{if }k_0(B)<j\le k(B)\end{cases}\]
for $i=0,\ldots,2^{r-1}-1$, where $\chi_1,\ldots,\chi_{k_0(B)}$ are the characters of height $0$.

Now let $\IBr(b_c):=\{\phi_1,\phi_2\}$. We determine the numbers $d_{\chi\phi_1}^c,d_{\chi\phi_2}^c\in\mathbb{Z}[\zeta_{2^{r-1}}]$. By (4C) in \cite{BrauerBlSec2} we have $d_{\chi\phi_1}^c\ne 0$ or $d_{\chi\phi_2}^c\ne 0$ for all $\chi\in\Irr(B)$. As in the proof of Theorem~\ref{s1klki}, $b_c$ dominates a block $\overline{b_c}\in\Bl(R[\C_G(c)/\langle c\rangle])$ with defect group $D_8$. The table at the end of \cite{Erdmann} shows that the Cartan matrix of $\overline{b_c}$ has the form
\[\begin{pmatrix}8&4\\4&3\end{pmatrix}\text{ or }\begin{pmatrix}4&2\\2&3\end{pmatrix}.\]
We label these possibilities as the “first” and the “second” case. The Cartan matrix of $b_c$ is 
\[2^{r-1}\begin{pmatrix}8&4\\4&3\end{pmatrix}\text{ or }2^{r-1}\begin{pmatrix}4&2\\2&3\end{pmatrix}\]
respectively.
The inverses of these matrices are
\[2^{-r-2}\begin{pmatrix}3&-4\\-4&8\end{pmatrix}\text{ and }2^{-r-2}\begin{pmatrix}3&-2\\-2&4\end{pmatrix}.\]

Let $m_{\chi\psi}^{(c,b_c)}$ be the contribution of $\chi,\psi\in\Irr(B)$ with respect to the subsection $(c,b_c)$ (see \cite{BrauerBlSec2}). Then we have
\begin{align}\label{ersterfall}
|D|m_{\chi\psi}^{(c,b_c)}&=3d_{\chi\phi_1}^c\overline{d_{\psi\phi_1}^c}-4(d_{\chi\phi_1}^c\overline{d_{\psi\phi_2}^c}+d_{\chi\phi_2}^c\overline{d_{\psi\phi_1}^c})+8d_{\chi\phi_2}^c\overline{d_{\psi\phi_2}^c}\nonumber\\
&\text{or}\nonumber\\
|D|m_{\chi\psi}^{(c,b_c)}&=3d_{\chi\phi_1}^c\overline{d_{\psi\phi_1}^c}-2(d_{\chi\phi_1}^c\overline{d_{\psi\phi_2}^c}+d_{\chi\phi_2}^c\overline{d_{\psi\phi_1}^c})+4d_{\chi\phi_2}^c\overline{d_{\psi\phi_2}^c}
\end{align}
respectively.
For a character $\chi\in\Irr(B)$ with height $0$ we get
\[0=h(\chi)=\nu\bigl(|D|m_{\chi\chi}^{(c,b_c)}\bigr)=\nu(3d_{\chi\phi_1}^c\overline{d_{\chi\phi_1}^c})=\nu(d_{\chi\phi_1}^c)\]
by (5H) in \cite{BrauerBlSec2}. In particular $d_{\chi\phi_1}^c\ne 0$. We define $c_i^j\in\mathbb{Z}^{k(B)}$ by
\[d_{\chi\phi_j}^c=\sum_{i=0}^{2^{r-2}-1}{c_i^j(\chi)\zeta_{2^{r-1}}^i}\]
for $j=1,2$. Then
\[(c_i^1,c_j^1)=\begin{cases}\delta_{ij}16&\text{first case}\\\delta_{ij}8&\text{second case}\end{cases},\ (c_i^1,c_j^2)=\begin{cases}\delta_{ij}8&\text{first case}\\\delta_{ij}4&\text{second case}\end{cases},\ (c_i^2,c_j^2)=\delta_{ij}6\]
as in Lemma~\ref{orthogonal}. (Since the $2^{r-2}$ $B$-subsections of the form $^\gamma(c,b_c)$ for $\gamma\in\mathcal{G}$ are pairwise nonconjugate, one can argue like in Lemma~\ref{orthogonal}.) Hence, in the second case 
\[d_{\chi_i\phi_1}^c=\begin{cases}\pm\zeta_{2^{r-1}}^{[\frac{i-1}{8}]}&\text{if }1\le i\le k_0(B)\\0&\text{if }k_0(B)<i\le k(B)\end{cases}\text{ (second case)}\]
holds for a suitable arrangement. Again $\chi_1,\ldots,\chi_{k_0(B)}$ are the characters of height $0$.
In the first case 
\[1=h(\psi)=\nu\bigl(|D|m_{\chi\psi}^{(c,b_c)}\bigr)=\nu(3d_{\chi\phi_1}^c\overline{d_{\psi\phi_1}^c})=\nu(d_{\psi\phi_1}^c)\]
by (5G) in \cite{BrauerBlSec2} for $h(\psi)=1$ and $h(\chi)=0$. As in Lemma~\ref{neukirchapp222} we also have $2\mid c_i^1(\psi)$ for $h(\psi)=1$ and $i=0,\ldots,2^{r-2}-1$. Analogously as in the case $u\in\Z(D)\setminus\langle c\rangle$ we conclude
\begin{equation}\label{zwfall}
d_{\chi\phi_1}^c=\begin{cases}\pm\zeta_{2^{r-1}}^{i(\chi)}&\text{if }h(\chi)=0\\\pm2\zeta_{2^{r-1}}^{i(\chi)}&\text{if }h(\chi)=1\end{cases}\text{ (first case)}
\end{equation}
for suitable indices $i(\chi)\in\{0,\ldots,2^{r-2}-1\}$. Since $(c_i^2,c_j^2)=\delta_{ij}6$, in both cases $c_i^2$ has the form
\[(\pm1,\pm1,\pm1,\pm1,\pm1,\pm1,0,\ldots,0)^{\text{T}}\text{ or }(\pm2,\pm1,\pm1,0,\ldots,0)^{\text{T}}.\]
We show that the latter possibility does not occur.
In the second case for every character $\chi\in\Irr(B)$ with height $1$ there exists $i\in\{0,\ldots,2^{r-2}-1\}$ such that $c_i^2(\chi)\ne 0$. In this case we get 
\[d_{\chi_i\phi_2}^c=\begin{cases}\pm\zeta_{2^{r-1}}^{[\frac{i-1}{4}]}&\text{if }1\le i\le 2^r\\0&\text{if }2^r<i\le k_0(B)\\\pm\zeta_{2^{r-1}}^{[\frac{i-k_0(B)-1}{2}]}&\text{if }k_0(B)<i\le k(B)\end{cases}\text{ (second case)},\]
where $\chi_1,\ldots,\chi_{k_0(B)}$ are again the characters of height $0$. Now let us consider the first case. Since $(c_i^1,c_j^2)=\delta_{ij}8$, the value $\pm2$ must occur in every column $c_i^1$ for $i=0,\ldots,2^{r-2}-1$ at least twice. Obviously exactly two entries have to be $\pm2$. Thus, one can improve equation~\eqref{zwfall} to
\[d_{\chi_i\phi_1}^c=\begin{cases}\pm\zeta_{2^{r-1}}^{[\frac{i-1}{8}]}&\text{if }1\le i\le k_0(B)\\\pm2\zeta_{2^{r-1}}^{[\frac{i-k_0(B)-1}{2}]}&\text{if }k_0(B)<i\le k(B)\end{cases}\text{ (first case)}.\]
It follows
\[d_{\chi_i\phi_2}^c=\begin{cases}\pm\zeta_{2^{r-1}}^{[\frac{i-1}{4}]}&\text{if }1\le i\le 2^r\\0&\text{if }2^r<i\le k_0(B)\\\pm\zeta_{2^{r-1}}^{[\frac{i-k_0(B)-1}{2}]}&\text{if }k_0(B)<i\le k(B)\end{cases}\text{ (first case)}.\]
Hence, the numbers $d^c_{\chi\phi_2}$ are independent of the case.
Of course, one gets similar results for $d_{\chi\phi_i}^u$ with $\langle u\rangle=\langle c\rangle$. 

\subsection{The Cartan matrix}
Now we investigate the Cartan matrix of $B$.
\begin{Lemma}\label{eledivs}
The elementary divisors of the Cartan matrix of $B$ are $2^{r-1}$ and $|D|$.
\end{Lemma}
\begin{proof}
Let $C$ be the Cartan matrix of $B$. Since $l(B)=2$, it suffices to show that $2^{r-1}$ occurs as elementary divisor of $C$ at least once. In order to proof this, we use the notion of lower defect groups (see \cite{OlssonLDG}). Let $(u,b)$ be a $B$-subsection with $|\langle u\rangle|=2^{r-1}$ and $l(b)=2$. Let $b_1:=b^{\N_G(\langle u\rangle)}$. Then $b_1$ has also defect group $D$, and $l(b_1)=2$ holds. Moreover, $u^{2^{r-2}}\in\Z(\N_G(\langle u\rangle))$. 
Let $\overline{b_1}\in\Bl(R[\N_G(u)/\langle u^{2^{r-2}}\rangle])$ be the block which is covered by $b_1$. Then $\overline{b_1}$ has defect group $D/\langle u^{2^{r-2}}\rangle$. We argue by induction on $r$. Thus, let $r=2$. Then $b=b_1$ and $D/\langle u^{2^{r-2}}\rangle=D/\langle u\rangle\cong D_8$. By Proposition~(5G) in \cite{Brauer} the Cartan matrix of $\overline{b}$ has the elementary divisors $1$ and $8$. Hence, $2=2^{r-1}$ and $16=|D|$ are the elementary divisors of the Cartan matrix of $b$. Hence, the claim follows from Theorem~7.2 in \cite{OlssonLDG}.

Now assume that the claim already holds for $r-1\ge 2$. By induction the elementary divisors of the Cartan matrix of $\overline{b_1}$ are $2^{r-2}$ and $|D|/2$. The claim follows easily as before.
\end{proof}

Now we are in a position to calculate the Cartan matrix $C$ up to equivalence of quadratic forms. Here we call two matrices $M_1,M_2\in\mathbb{Z}^{l\times l}$ equivalent if there exists a matrix $S\in\GL(l,\mathbb{Z})$ such that $A=SBS^{\text{T}}$, where $S^{\text{T}}$ denotes the transpose of $S$.

By Lemma~\ref{eledivs} all entries of $C$ are divisible by $2^{r-1}$. Thus, we can consider $\widetilde{C}:=2^{1-r}C\in\mathbb{Z}^{2\times 2}$. Then $\det\widetilde{C}=8$ and the elementary divisors of $\widetilde{C}$ are $1$ and $8$. If we write
\[\widetilde{C}=\begin{pmatrix}c_1&c_2\\c_2&c_3\end{pmatrix},\]
then $\widetilde{C}$ corresponds to the positive definite binary quadratic form $q(x_1,x_2):=c_1x_1^2+2c_2x_1x_2+c_3x_2^2$. 
Obviously $\gcd(c_1,c_2,c_3)=1$. If one reduces the entries of $\widetilde{C}$ modulo $2$, then one gets a matrix of rank $1$ (this is just the multiplicity of the elementary divisor $1$). This shows that $c_1$ or $c_3$ must be odd. Hence, $\gcd(c_1,2c_2,c_3)=1$, i.\,e. $q$ is primitive (see \cite{Buell} for example). Moreover, $\Delta:=-4\det\widetilde{C}=-32$ is the discriminant of $q$. 
Now it is easy to see that $q$ (and $\widetilde{C}$) is equivalent to exactly one of the following matrices (see page~20 in \cite{Buell}):
\[\begin{pmatrix}1&0\\0&8\end{pmatrix}\text{ or }\begin{pmatrix}3&1\\1&3\end{pmatrix}.\]

The Cartan matrices for the block $\overline{b_c}$ with defect group $D_8$ (used before) satisfy
\[\begin{pmatrix}1&-1\\0&1\end{pmatrix}\begin{pmatrix}8&4\\4&3\end{pmatrix}\begin{pmatrix}1&-1\\0&1\end{pmatrix}^{\text{T}}=\begin{pmatrix}0&1\\-1&1\end{pmatrix}\begin{pmatrix}4&2\\2&3\end{pmatrix}\begin{pmatrix}0&1\\-1&1\end{pmatrix}^{\text{T}}=\begin{pmatrix}3&1\\1&3\end{pmatrix}.\]
Hence, only the second matrix occurs up to equivalence. We show that this holds also for the block $B$. 

\begin{Theorem}
The Cartan matrix of $B$ is equivalent to
\[2^{r-1}\begin{pmatrix}3&1\\1&3\end{pmatrix}.\]
\end{Theorem}
\begin{proof}
We argue by induction on $r$. The smallest case was already considered by $b_c$ (this would correspond to $r=1$). Thus, we may assume $r\ge 2$ (as usual). First, we determine the generalized decomposition numbers $d^u_{\chi\phi}$ for $u\in\langle c\rangle\setminus\{1\}$ with $|\langle u\rangle|=2^k<2^{r-1}$. As in the proof of Theorem~\ref{s1klki}, the group $D/\langle u\rangle$ has the same isomorphism type as $D$, but one has to replace $r$ by $r-k$. Hence, by induction we may assume that $b_u$ has a Cartan matrix which is equivalent to the matrix given in the statement of the theorem. Let $C_u$ be the Cartan matrix of $b_u$, and let $S_u\in\GL(2,\mathbb{Z})$ such that
\[C_u=2^{r-1}S_u^{\text{T}}\begin{pmatrix}4&2\\2&3\end{pmatrix}S_u,\]
i.\,e. with the notations of the previous section, we assume that the “second case” occurs. (This is allowed, since we can only compute the generalized decomposition numbers up to multiplication with $S_u$ anyway.) As before we write $\IBr(b_u)=\{\phi_1,\phi_2\}$, $D_u:=(d^u_{\chi\phi_i})$ and $(\widetilde{d}^u_{\chi\phi_i}):=D_uS_u^{-1}$.
The consideration in the previous section carries over, and one gets
\[\widetilde{d}^u_{\chi\phi_1}=\begin{cases}\pm\zeta_{2^k}^{[\frac{i-1}{2^{r+2-k}}]}&\text{if }1\le i\le k_0(B)\\0&\text{if }k_0(B)<i\le k(B)\end{cases}\]
and
\[\widetilde{d}^u_{\chi\phi_2}=\begin{cases}\pm\zeta_{2^k}^{[\frac{i-1}{2^{r-k+1}}]}&\text{if }1\le i\le 2^r\\0&\text{if }2^r<i\le k_0(B)\\\pm\zeta_{2^k}^{[\frac{i-k_0(B)-1}{2^{r-k}}]}&\text{if }k_0(B)<i\le k(B)\end{cases},\]
where $\chi_1,\ldots,\chi_{k_0(B)}$ are the characters of height $0$. But notice that the ordering of those characters for $\phi_1$ and $\phi_2$ is different.

Now assume that there is a matrix $S\in\GL(2,\mathbb{Z})$ such that
\[C=2^{r-1}S^{\text{T}}\begin{pmatrix}1&0\\0&8\end{pmatrix}S.\]
If $Q$ denotes the decomposition matrix of $B$, we set $(\widetilde{d}_{\chi\phi_i}):=QS^{-1}$ for $\IBr(B)=\{\phi_1,\phi_2\}$. Then we have
\[|D|m_{\chi\psi}^{(1,B)}=8\widetilde{d}_{\chi\phi_1}\widetilde{d}_{\psi\phi_1}+\widetilde{d}_{\chi\phi_2}\widetilde{d}_{\psi\phi_2}\text{ for }\chi,\psi\in\Irr(B).\]
In particular $|D|m_{\chi\chi}^{(1,B)}\equiv 1\pmod{4}$ for a character $\chi\in\Irr(B)$ of height $0$. For $u\in\mathcal{T}\setminus\Z(D)$ we have $|D|m_{\chi\chi}^{(u,b_u)}=2$, and for $u\in\Z(D)\setminus\langle c\rangle$ we have $|D|m_{\chi\chi}^{(u,b_u)}=1$. Let $u\in\langle c\rangle\setminus\{1\}$. Equation~\eqref{ersterfall} and the considerations above imply $|D|m_{\chi\chi}^{(u,b_u)}\equiv 3\pmod{4}$. Now (5B) in \cite{BrauerBlSec2} reveals the contradiction
\[|D|=\sum_{u\in\mathcal{T}}{|D|m_{\chi\chi}^{(u,b_u)}}\equiv|D|m_{\chi\chi}^{(1,B)}+2^{r+1}+2^{r-1}+3\cdot(2^{r-1}-1)\equiv 2\pmod{4}.\qedhere\]
\end{proof}

With the proof of the last theorem we can also obtain the ordinary decomposition numbers (up to multiplication with an invertible matrix):
\begin{align*}
&d_{\chi\phi_1}=\begin{cases}\pm1&\text{if }h(\chi)=0\\0&\text{if }h(\chi)=1\end{cases},
&d_{\chi_i\phi_2}=\begin{cases}\pm1&\text{if }0\le i\le 2^r\\0&\text{if }2^r<i\le k_0(B)\\\pm1&\text{if }k_0(B)<i\le k(B)\end{cases}.
\end{align*}
Again $\chi_1,\ldots,\chi_{k_0(B)}$ are the characters of height $0$.

Since we know how $\mathcal{G}$ acts on the $B$-subsections, we can investigate the action of $\mathcal{G}$ on $\Irr(B)$.

\begin{Theorem}
The irreducible characters of height $0$ of $B$ split in $2(r+1)$ families of $2$-conjugate characters. These families have sizes $1,1,1,1,2,2,4,4,\ldots,2^{r-1},2^{r-1}$ respectively. The characters of height $1$ split in $r$ families with sizes $1,1,2,4,\ldots,2^{r-2}$ respectively. In particular there are exactly six $2$-rational characters in $\Irr(B)$. 
\end{Theorem}
\begin{proof}
We start by determining the number of orbits of the action of $\mathcal{G}$ on the columns of the generalized decomposition matrix. The columns $\{d^u_{\chi\phi_u}:\chi\in\Irr(B)\}$ with $u\in\mathcal{T}\setminus\Z(D)$ split in two orbits of length $2^{r-1}$.
For $i=1,2$ the columns $\{d^u_{\chi\phi_i}:\chi\in\Irr(B)\}$ with $u\in\langle c\rangle$ split in $r$ orbits of lengths $1,1,2,4,\ldots,2^{r-2}$ respectively.
Finally, the columns $\{d^u_{\chi\phi_u}:\chi\in\Irr(B)\}$ with $u\in\Z(D)\setminus\langle c\rangle$ consist of $r$ orbits of lengths $1,1,2,4,\ldots,2^{r-2}$ respectively. This gives $3r+2$ orbits altogether. By Theorem~11 in \cite{Brauerconnection} there also exist exactly $3r+2$ families of $2$-conjugate characters. (Since $\mathcal{G}$ is noncyclic, one cannot conclude a priori that also the lengths of the orbits of these two actions coincide.) 

By considering the column $\{d^x_{\chi\phi_x}:\chi\in\Irr(B)\}$, we see that the irreducible characters of height $0$ split in at most $2(r+1)$ orbits of lengths $1,1,1,1,2,2,4,4,\ldots,2^{r-1},2^{r-1}$ respectively. Similarly the column $\{d^c_{\chi\phi_2}:\chi\in\Irr(B)\}$ shows that there are at most $r$ orbits of lengths $1,1,2,4,\ldots,2^{r-2}$ of characters of height $1$. Since $2(r+1)+r=3r+2$, these orbits do not merge further, and the claim is proved.
\end{proof}

Let $M=\langle x^2,y,z\rangle$ as in Lemma~\ref{lokallb}. Then $D\subseteq\inertial_{\N_G(M)}(b_M)$. Since $e(B)=1$, Alperin's fusion theorem implies that $\inertial_{\N_G(M)}(b_M)$ controls the fusion of $B$-subpairs. By Lemma~\ref{lokallb} we also have $\inertial_{\N_G(M)}(b_M)\subseteq\C_G(c)$ for a $c\in\Z(D)$. This shows that $B$ is a so called “centrally controlled block” (see \cite{KuelshammerOkuyama}). In \cite{KuelshammerOkuyama} it was shown that then the centers of the blocks $B$ and $b_c$ (regarded as blocks of $FG$) are isomorphic. 

\subsection{Dade's conjecture}
In this section we will verify Dade's (ordinary) conjecture for the block $B$ (see \cite{Dadeconj}). First, we need a lemma.

\begin{Lemma}\label{eB3}
Let $\widetilde{B}$ be a block of $RG$ with defect group $\widetilde{D}\cong C_{2^s}\times C_2^2$ ($s\in\mathbb{N}_0$) and inertial index $3$. Then $k(\widetilde{B})=k_0(\widetilde{B})=|\widetilde{D}|=2^{s+2}$ and $l(\widetilde{B})=3$ hold.
\end{Lemma}
\begin{proof}
Let $\alpha$ be an automorphism of $\widetilde{D}$ of order $3$ which is induced by the inertial group. By Lemma~\ref{auts1} we have $\C_{\widetilde{D}}(\alpha)\cong C_{2^s}$. We choose a system of representatives $x_1,\ldots,x_k$ for the orbits of $\widetilde{D}\setminus\C_{\widetilde{D}}(\alpha)$ under $\alpha$. Then $k=2^s$. If $b_i\in\Bl(R\C_G(x_i))$ for $i=1,\ldots,k$ and $b_u\in\Bl(R\C_G(u))$ for $u\in\C_{\widetilde{D}}(\alpha)$ are Brauer correspondents of $\widetilde{B}$, then
\[\bigcup_{i=1}^k{\bigl\{(x_i,b_i)\bigr\}}\ \cup\bigcup_{u\in\C_{\widetilde{D}}(\alpha)}{\bigl\{(u,b_u)\bigr\}}\]
is a system of representatives for the conjugacy classes of $\widetilde{B}$-subsections.
Since $\alpha\notin\C_G(x_i)$, we have $l(b_i)=1$ for $i=1,\ldots,k$. In particular $k(\widetilde{B})\le 2^{s+2}$ holds. Now we show the opposite inequality by induction on $s$. 

For $s=0$ the claim is well known. Let $s\ge 1$. By induction $l(b_u)=3$ for $u\in\C_{\widetilde{D}}(\alpha)\setminus\{1\}$. This shows $k(\widetilde{B})-l(\widetilde{B})=k+(2^s-1)3=2^{s+2}-3$ and $l(\widetilde{B})\le 3$. An inspection of the numbers $d^{x_1}_{\chi\phi}$ implies  $k(\widetilde{B})=k_0(\widetilde{B})=2^{s+2}=|\widetilde{D}|$ and $l(\widetilde{B})=3$. (This would also follow from Theorem~1 in \cite{Watanabe1}.)
\end{proof}

Now assume $\pcore_2(G)=1$ (this is a hypothesis of Dade's conjecture). In order to prove Dade's conjecture it suffices to consider chains
\[\sigma:P_1<P_2<\ldots<P_n\]
of nontrivial elementary abelian $2$-subgroups of $G$ (see \cite{Dadeconj}). (Note that also the empty chain is allowed.)
In particular $P_i\unlhd P_n$ and $P_n\unlhd\N_G(\sigma)$ for $i=1,\ldots,n$. Hence, for a block $b\in\Bl(R\N_G(\sigma))$ with $b^G=B$ and defect group $Q$ we have $P_n\le Q$. Moreover, there exists a $g\in G$ such that $^gQ\le D$. Thus, by conjugation with $g$ we may assume $P_n\le Q\le D$ (see also Lemma~6.9 in \cite{Dadeconj}). This shows $n\le 3$. 

In the case $|P_n|=8$ we have $P_n=\langle x^{2^{r-1}},y,z\rangle=:E$, because this is the only elementary abelian subgroup of order $8$ in $D$. Let $b\in\Bl(R\N_G(\sigma))$ with $b^G=B$. We choose a defect group $Q$ of $\widetilde{B}:=b^{\N_G(E)}$. Since $\Omega(Q)=P_n$, we get $\N_G(Q)\le\N_G(E)$. Then Brauer's first main theorem implies $Q=D$. Hence, $\widetilde{B}$ is the unique Brauer correspondent of $B$ in $R\N_G(E)$. For $M:=\langle x^2,y,z\rangle\le D$ we also have $\N_G(M)\le\N_G(\Omega(M))=\N_G(E)$. Hence, $\widetilde{B}$ is nonnilpotent. 
Now consider the chain 
\[\widetilde{\sigma}:\begin{cases}\varnothing&\text{if }n=1\\P_1&\text{if }n=2\\P_1<P_2&\text{if }n=3\end{cases}\] 
for the group $\widetilde{G}:=\N_G(E)$. Then $\N_G(\sigma)=\N_{\widetilde{G}}(\widetilde{\sigma})$ and
\[\sum_{\substack{b\in\Bl(R\N_G(\sigma)),\\b^G=B}}{k^i(b)}=\sum_{\substack{b\in\Bl(R\N_{\widetilde{G}}(\widetilde{\sigma})),\\b^{\widetilde{G}}=\widetilde{B}}}{k^i(b)}.\]
The chains $\sigma$ and $\widetilde{\sigma}$ account for all possible chains of $G$. Moreover, the lengths of $\sigma$ and $\widetilde{\sigma}$ have opposite parity. Thus, it seems plausible that the contributions of $\sigma$ and $\widetilde{\sigma}$ in the alternating sum cancel out each other (this would imply Dade's conjecture). The question which remains is: Can we replace $(\widetilde{G},\widetilde{B},\widetilde{\sigma})$ by $(G,B,\widetilde{\sigma})$? We make this more precise in the following lemma.

\begin{Lemma}
Let $\mathcal{Q}$ be a system of representatives for the $G$-conjugacy classes of pairs $(\sigma,b)$, where $\sigma$ is a chain (of $G$) of length $n$ with $P_n<E$ and $b\in\Bl(R\N_G(\sigma))$ is a Brauer correspondent of $B$. Similarly, let $\widetilde{\mathcal{Q}}$ be a system of representatives for the $\widetilde{G}$-conjugacy classes of pairs $(\widetilde{\sigma},\widetilde{b})$, where $\widetilde{\sigma}$ is a chain (of $\widetilde{G}$) of length $n$ with $P_n<E$ and $\widetilde{b}\in\Bl(R\N_{\widetilde{G}}(\widetilde{\sigma}))$ is a Brauer correspondent of $\widetilde{B}$. Then there exists a bijection between $\mathcal{Q}$ and $\widetilde{\mathcal{Q}}$ which preserves the numbers $k^i(b)$.
\end{Lemma}
\begin{proof}
Let $b_D\in\Bl(R\N_G(D))$ be a Brauer correspondent of $B$. We consider chains of $B$-subpairs
\[\sigma:(P_1,b_1)<(P_2,b_2)<\ldots<(P_n,b_n)<(D,b_D),\]
where the $P_i$ are nontrivial elementary abelian $2$-subgroups such that $P_n<E$. Then $\sigma$ is uniquely determined by these subgroups $P_1,\ldots,P_n$ (see Theorem~1.7 in \cite{Olssonsubpairs}). Moreover, the empty chain is also allowed. Let $\mathcal{U}$ be a system of representatives for $G$-conjugacy classes of such chains. For every chain $\sigma\in\mathcal{U}$ we define
\[\widetilde{\sigma}:(P_1,\widetilde{b_1})<(P_2,\widetilde{b_2})<\ldots<(P_n,\widetilde{b_n})<(D,b_D)\]
with $\widetilde{b_i}\in\Bl(R\C_{\widetilde{G}}(P_i))$ for $i=1,\ldots,n$. Finally we set $\widetilde{\mathcal{U}}:=\{\widetilde{\sigma}:\sigma\in\mathcal{U}\}$.
By Alperin's fusion theorem $\widetilde{\mathcal{U}}$ is a system of representatives for the $\widetilde{G}$-conjugacy classes of corresponding chains for the group $\widetilde{B}$. Hence, it suffices to show the existence of bijections $f$ (resp. $\widetilde{f}$) between $\mathcal{U}$ (resp. $\widetilde{\mathcal{U}}$) and $\mathcal{Q}$ (resp. $\widetilde{\mathcal{Q}}$) such that the following property is satisfied:
If $f(\sigma)=(\tau,b)$ and $\widetilde{f}(\widetilde{\sigma})=(\widetilde{\tau},\widetilde{b})$, then $k^i(b)=k^i(\widetilde{b})$ for all $i\in\mathbb{N}_0$.

Let $\sigma\in\mathcal{U}$. Then we define the chain $\tau$ by only considering the subgroups of $\sigma$, i.\,e. $\tau:P_1<\ldots<P_n$. This gives $\C_G(P_n)\subseteq\N_G(\tau)$, and we can define
\[f:\mathcal{U}\to\mathcal{Q},\ \sigma\mapsto\bigl(\tau,b_n^{\N_G(\tau)}\bigr).\]
Now let $(\sigma,b)\in\mathcal{Q}$ arbitrary. We write $\sigma:P_1<\ldots<P_n$. By Theorem~5.5.15 in \cite{Nagao} there exists a Brauer correspondent $\beta_n\in\Bl(R\C_G(P_n))$ of $b$. Since $(P_n,\beta_n)$ is a $B$-subpair, we may assume $(P_n,\beta_n)<(D,b_D)$ after a suitable conjugation. Then there are uniquely determined blocks $\beta_i\in\Bl(R\C_G(P_i))$ for $i=1,\ldots,n-1$ such that
\[(P_1,\beta_1)<(P_2,\beta_2)<\ldots<(P_n,\beta_n)<(D,b_D).\]

This shows that $f$ is surjective.

Now let $\sigma_1,\sigma_2\in\mathcal{U}$ be given. We write
\[\sigma_i:(P_1^i,\beta_1^i)<\ldots<(P_n^i,\beta_n^i)\]
for $i=1,2$. Let us assume that $f(\sigma_1)=(\tau_1,b_1)$ and $f(\sigma_2)=(\tau_2,b_2)$ are conjugate in $G$, i.\,e. there is a $g\in G$ such that
\[\bigl(\tau_2,({^g\beta_n^1})^{\N_G(\tau_2)}\bigr)={^g(\tau_1,b_1)}=(\tau_2,b_2)=\bigl(\tau_2,(\beta_n^2)^{\N_G(\tau_2)}\bigr).\]
Since $^g\beta_n^1\in\Bl(R\C_G(P_n^2))$ and $\beta_n^2$ are covered by $b_2$, there is $h\in\N_G(\tau_2)$ with $^{hg}\beta_n^1=\beta_n^2$. Then
\[^{hg}(P_n^1,\beta_n^1)=(P_n^2,\beta_n^2).\]
Since the blocks $\beta_j^i$ for $i=1,2$ and $j=1,\ldots,n-1$ are uniquely determined by $P_j^i$, we also have $^{gh}\sigma_1=\sigma_2=\sigma_1$. This proves the injectivity of $f$. Analogously, we define the map $\widetilde{f}$.

It remains to show that $f$ and $\widetilde{f}$ satisfy the property given above. For this let $\sigma\in\mathcal{U}$ with $\sigma:(P_1,b_1)<\ldots<(P_n,b_n)$, $\widetilde{\sigma}:(P_1,\widetilde{b_1})<\ldots<(P_n,\widetilde{b_n})$, $f(\sigma)=\bigl(\tau,b_n^{\N_G(\tau)}\bigr)$ and $\widetilde{f}(\widetilde{\sigma})=\bigl(\tau,\widetilde{b_n}^{\N_{\widetilde{G}}(\tau)}\bigr)$. We have to prove $k^i\bigl(b_n^{\N_G(\tau)}\bigr)=k^i\bigl(\widetilde{b_n}^{\N_{\widetilde{G}}(\tau)}\bigr)$ for $i\in\mathbb{N}_0$.

Let $Q$ be a defect group of $b_n^{\N_G(\tau)}$. Then $Q\C_G(Q)\subseteq\N_G(\tau)$, and there is a Brauer correspondent $\beta_n\in\Bl(RQ\C_G(Q))$ of $b_n^{\N_G(\tau)}$. In particular $(Q,\beta_n)$ is a $B$-Brauer subpair.
As in Lemma~\ref{repkonjs1} we may assume $Q\in\{D,M,\langle x,z\rangle,\langle xy,z\rangle\}$.
The same considerations also work for the defect group $\widetilde{Q}$ of $\widetilde{b_n}^{\N_{\widetilde{G}}(\tau)}$. Since $b_n^{D\C_G(P_n)}=b_D^{D\C_G(P_n)}=\widetilde{b_n}^{D\C_G(P_n)}$, we get:
\[Q=D\Longleftrightarrow D\subseteq\N_G(\tau)\Longleftrightarrow D\subseteq\N_{\widetilde{G}}(\tau)\Longleftrightarrow\widetilde{Q}=D.\]

Let us consider the case $Q=D\ (\,=\widetilde{Q})$. Let $b_M\in\Bl(R\C_G(M))$ such that $(M,b_M)\le(D,b_D)$ and $\alpha\in\inertial_{\N_G(M)}(b_M)\setminus D\C_G(M)\subseteq\N_G(M)\subseteq\widetilde{G}$. Then:
\[b_n^{\N_G(\tau)}\text{ is nilpotent}\Longleftrightarrow\alpha\notin\N_G(\tau)\Longleftrightarrow\alpha\notin\N_{\widetilde{G}}(\tau)\Longleftrightarrow\widetilde{b_n}^{\N_{\widetilde{G}}(\tau)}\text{ is nilpotent}.\]
Thus, the claim holds in this case. Now let $Q<D$ (and $\widetilde{Q}<D$). Then we have $Q\C_G(Q)=\C_G(Q)\subseteq\C_G(P_n)$.
Since $\beta_n^{\C_G(P_n)}$ is also a Brauer correspondent of $b_n^{\N_G(\tau)}$, the blocks $\beta_n^{\C_G(P_n)}$ and $b_n$ are conjugate. In particular $b_n$ (and $\widetilde{b_n}$) has defect group $Q$. Hence, we obtain $Q=\widetilde{Q}$. If $Q\in\{\langle x,z\rangle,\langle xy,z\rangle\}$, then $b_n^{\N_G(\tau)}$ and $\widetilde{b_n}^{\N_{\widetilde{G}}(\tau)}$ are nilpotent, and the claim holds. Thus, we may assume $Q=M$. Then as before:
\[b_n^{\N_G(\tau)}\text{ is nilpotent}\Longleftrightarrow\alpha\notin\N_G(\tau)\Longleftrightarrow\alpha\notin\N_{\widetilde{G}}(\tau)\Longleftrightarrow\widetilde{b_n}^{\N_{\widetilde{G}}(\tau)}\text{ is nilpotent}.\]
We may assume that the nonnilpotent case occurs. Then $t\bigl(b_n^{\N_G(\tau)}\bigr)=t\bigl(\widetilde{b_n}^{\N_{\widetilde{G}}(\tau)}\bigr)=3$, and the claim follows from Lemma~\ref{eB3}.
\end{proof}

As explained in the beginning of the section, the Dade conjecture follows.

\begin{Theorem}
The Dade conjecture holds for $B$.
\end{Theorem}

\subsection{Alperin's weight conjecture}
In this section we prove Alperin's weight conjecture for $B$.
Let $(P,\beta)$ be a weight for $B$, i.\,e. $P$ is a $2$-subgroup of $G$ and $\beta$ is a block of $R[\N_G(P)/P]$ with defect $0$. Moreover, $\beta$ is dominated by a Brauer correspondent $b\in\Bl(R\N_G(P))$ of $B$. As usual, one can assume $P\le D$. If $\Aut(P)$ is a $2$-group, then $\N_G(P)/\C_G(P)$ is also a $2$-group. Then $P$ is a defect group of $b$, since $\beta$ has defect $0$. Moreover, $\beta$ is uniquely determined by $b$. By Brauer's first main theorem we have $P=D$. Thus, in this case there is exactly one weight for $B$ up to conjugation.

Now let us assume that $\Aut(P)$ is not a $2$-group (in particular $P<D$). As usual, $\beta$ covers a block $\beta_1\in\Bl(R[\C_G(P)/P])$. By the Fong-Reynolds theorem (see \cite{Nagao} for example) also $\beta_1$ has defect $0$. Hence, $\beta_1$ is dominated by exactly one block $b_1\in\Bl(R\C_G(P))$ with defect group $P$. Since $\beta\beta_1\ne 0$, we also have $bb_1\ne 0$, i.\,e. $b$ covers $b_1$. Thus, the situation is as follows:
\[\begin{xy}
\xymatrix{\beta\in\Bl(R[\N_G(P)/P])\ar@{<->}[r]\ar@{<->}[d]&b\in\Bl(R\N_G(P))\ar@{<->}[d]\\\beta_1\in\Bl(R[\C_G(P)/P])\ar@{<->}[r]&b_1\in\Bl(R\C_G(P))}
\end{xy}\]
By Theorem~5.5.15 in \cite{Nagao} we have $b_1^{\N_G(P)}=b$ and $b_1^G=B$. This shows that $(P,b_1)$ is a $B$-Brauer subpair. Then $P=M\ (\,=\!\langle x^2,y,z\rangle)$ follows. By Brauer's first main theorem $b$ is uniquely determined (independent of $\beta$). Now we prove that also $\beta$ is uniquely determined by $b$. 

In order to do so it suffices to show that $\beta$ is the only block with defect $0$ which covers $\beta_1$. By the Fong-Reynolds theorem it suffices to show that $\beta_1$ is covered by only one block of $R\inertial_{\N_G(M)/M}(\beta_1)=R[\inertial_{\N_G(M)}(b_1)/M]$ with defect $0$. For convenience we write $\overline{\C_G(M)}:=\C_G(M)/M$, $\overline{\N_G(M)}:=\N_G(M)/M$ and $\overline{\inertial}:=\inertial_{\N_G(M)}(b_1)/M$. Let $\chi\in\Irr(\beta_1)$. The irreducible constituents of $\Ind^{\overline{\inertial}}_{\overline{\C_G(M)}}(\chi)$ belong to blocks which covers $\beta_1$ (where $\Ind$ denote induction). Conversely, every block of $R\overline{\inertial}$ which covers $\beta_1$ arises in this way (see Lemma~5.5.7 in \cite{Nagao}). Let
\[\Ind^{\overline{\inertial}}_{\overline{\C_G(M)}}(\chi)=\sum_{i=1}^t{e_i\psi_i}\]
with $\psi_i\in\Irr(\overline{\inertial})$ and $e_i\in\mathbb{N}$ for $i=1,\ldots,t$. Then
\[\sum_{i=1}^t{e_i^2}=|\overline{\inertial}:\overline{\C_G(M)}|=|\inertial_{\N_G(M)}(b_1):\C_G(M)|=6\]
(see page~84 in \cite{Isaacs}).
Thus, there is some $i\in\{1,\ldots,t\}$ with $e_i=1$, i.\,e. $\chi$ is extendible to $\overline{\inertial}$. We may assume $e_1=1$. By Corollary~6.17 in \cite{Isaacs} it follows that $t=|\Irr(\overline{\inertial}/\overline{\C_G(M)})|=|\Irr(S_3)|=3$ and
\[\{\psi_1,\psi_2,\psi_3\}=\bigl\{\psi_1\tau:\tau\in\Irr(\overline{\inertial}/\overline{\C_G(M)})\bigr\},\]
where the characters in $\Irr(\overline{\inertial}/\overline{\C_G(M)})$ were identified with their inflations in $\Irr(\overline{\inertial})$. Thus, we may assume $e_2=1$ and $e_3=2$. Then it is easy to see that $\psi_1$ and $\psi_2$ belong to blocks with defect at least $1$. Hence, only the block with contains $\psi_3$ is allowed. This shows uniqueness.

Finally we show that there is in fact a weight of the form $(M,\beta)$. For this we choose $b$, $b_1$, $\beta_1$, $\chi$ and $\psi_i$ as above. Then $\chi$ vanishs on all nontrivial $2$-elements. Moreover, $\psi_1$ is an extension of $\chi$. Let $\tau\in\Irr(\overline{\inertial}/\overline{\C_G(M)})$ be the character of degree $2$. Then $\tau$ vanishs on all nontrivial $2$-elements of $\overline{\inertial}/\overline{\C_G(M)}$. Hence, $\psi_3=\psi_1\tau$ vanishs on all nontrivial $2$-elements of $\overline{\inertial}$. This shows that $\psi_3$ belongs in fact to a block $\widetilde{\beta}\in\Bl(R\overline{\inertial})$ with defect $0$. Then $\bigl(M,\widetilde{\beta}^{\,\overline{\N_G(M)}}\bigr)$ is the desired weight for $B$.

Hence, we have shown that there are exactly two weights for $B$ up to conjugation. Since $l(B)=2$, Alperin's weight conjecture is satisfied.

\begin{Theorem}
Alperin's weight conjecture holds for $B$.
\end{Theorem}

\subsection{The gluing problem}\label{gluginsubsec}
Finally we show that the gluing problem (see Conjecture~4.2 in \cite{gluingprob}) for the block $B$ has a unique solution. We will not recall the very technical statement of the gluing problem. Instead we refer to \cite{Parkgluing} for most of the  notations. Observe that the field $F$ is denoted by $k$ in \cite{Parkgluing}.

\begin{Theorem}
The gluing problem for $B$ has a unique solution.
\end{Theorem}
\begin{proof}
As in \cite{Parkgluing} we denote the fusion system induced by $B$ with $\mathcal{F}$. Then the $\mathcal{F}$-centric subgroups of $D$ are given by $M_1:=\langle x^2,y,z\rangle$, $M_2:=\langle x,z\rangle$, $M_3:=\langle xy,z\rangle$ and $D$. We have seen so far that $\Aut_{\mathcal{F}}(M_1)\cong\Out_{\mathcal{F}}(M_1)\cong S_3$, $\Aut_{\mathcal{F}}(M_i)\cong D/M_i\cong C_2$ for $i=2,3$ and $\Aut_{\mathcal{F}}(D)\cong D/\Z(D)\cong C_2^2$ (see proof of Lemma~\ref{lokallb}). Using this, we get $\cohom^i(\Aut_{\mathcal{F}}(\sigma),F^\times)=0$ for $i=1,2$ and every chain $\sigma$ of $\mathcal{F}$-centric subgroups (see proof of Corollary~2.2 in \cite{Parkgluing}). Hence, $\cohom^0([S(\mathcal{F}^c)],\mathcal{A}^2_{\mathcal{F}})=\cohom^1([S(\mathcal{F}^c)],\mathcal{A}^1_{\mathcal{F}})=0$. Now the claim follows from Theorem~1.1 in \cite{Parkgluing}.
\end{proof}

\section{The case $r=s>1$}
In the section we assume that $B$ is a nonnilpotent block of $RG$ with defect group
\[D:=\langle x,y\mid x^{2^r}=y^{2^r}=[x,y]^2=[x,x,y]=[y,x,y]=1\rangle\]
for $r\ge 2$. As before we define $z:=[x,y]$. Since $|D/\Phi(D)|=4$, $2$ and $3$ are the only prime divisors of $|\Aut(D)|$. In particular $t(B)\in\{1,3\}$. If $t(B)=1$, then $B$ would be nilpotent by Theorem~\ref{mnafusion}. Thus, we have $t(B)=3$.

\subsection{The $B$-subsections}
We investigate the automorphism group of $D$.
\begin{Lemma}\label{uniquefix}
Let $\alpha\in\Aut(D)$ be an automorphism of order $3$. Then $z$ is the only nontrivial fixed-point of $\Z(D)$ under $\alpha$.
\end{Lemma}
\begin{proof}
Since $D'=\langle z\rangle$, $z$ remains fixed under all automorphisms of $D$. Moreover, $\alpha(x)\in y\Z(D)\cup xy\Z(D)$, because $\alpha$ acts nontrivially on $D/\Z(D)$. In both cases we have $\alpha(x^2)\ne x^2$. This shows that $\alpha|_{\Z(D)}\in\Aut(\Z(D))$ is also an automorphism of order $3$. Obviously $\alpha$ induces an automorphism of order $3$ on $\Z(D)/\langle z\rangle\cong C_{2^{r-1}}^2$. But this automorphism is fixed-point-free (see Lemma~1 in \cite{Mazurov}). The claim follows. 
\end{proof}

Using this, we can find a system of representatives for the conjugacy classes of $B$-subsections.

\begin{Lemma}\label{repsystemrs}
Let $b\in\Bl(RD\C_G(D))$ be a Brauer correspondent of $B$, and for $Q\le D$ let $b_Q$ be the unique block of $RQ\C_G(Q)$ with $(Q,b_Q)\le (D,b)$. We choose a system $\mathcal{S}\subseteq\Z(D)$ of representatives for the orbits of $\Z(D)$ under the action of $\inertial_{\N_G(D)}(b)$. We set $\mathcal{T}:=\mathcal{S}\cup\{y^ix^{2j}:i,j\in\mathbb{Z},\ i\text{ odd}\}$. Then
\begin{equation*}
\bigcup_{a\in\mathcal{T}}{\Bigl\{\bigl(a,b_{\C_D(a)}^{\C_G(a)}\bigr)\Bigr\}}
\end{equation*}
is a system of representatives for the conjugacy classes of $B$-subsections. Moreover,
\[|\mathcal{T}|=\frac{5\cdot2^{2(r-1)}+4}{3}.\]
\end{Lemma}
\begin{proof}
Proposition~2.12.(ii) in \cite{OlssonRedei} states the desired system wrongly. More precisely the claim $I_D=\Z(D)$ in the proof is false. Indeed Lemma~\ref{uniquefix} shows $I_D=\mathcal{S}$. Now the claim follows easily. 
\end{proof}

From now on we write $b_a:=b_{\C_D(a)}^{\C_G(a)}$ for $a\in\mathcal{T}$.
We are able to determine the difference $k(B)-l(B)$.

\begin{Proposition}\label{kminusl}
We have
\[k(B)-l(B)=\frac{5\cdot 2^{2(r-1)}+7}{3}.\]
\end{Proposition}
\begin{proof}
Consider $l(b_a)$ for $1\ne a\in\mathcal{T}$. 

\textbf{Case 1:} $a\in\Z(D)$. \\
Then $b_a$ is a block with defect group $D$. Moreover, $b_a$ and $B$ have a common Brauer correspondent in \linebreak$\Bl(RD\C_{\C_G(a)}(D))=\Bl(RD\C_G(D))$. In case $a\ne z$ we have $t(b_a)=1$ by Lemma~\ref{uniquefix}. Hence, $b_a$ is nilpotent and $l(b_a)=1$. Now let $a=z$. Then there exists a block $\overline{b_z}$ of $\C_G(z)/\langle z\rangle$ with defect group $D/\langle z\rangle\cong C_{2^r}^2$ and $l(\overline{b_z})=l(b_z)$. By Theorem~1.5(iv) in \cite{Olsson}, $t(\overline{b_z})=t(b_z)=3$ holds. Thus, Theorem~2 in \cite{Sambale} implies $l(b_z)=l(\overline{b_z})=3$.

\textbf{Case 2:} $a\notin\Z(D)$.\\
Then $b_{\C_P(a)}=b_M$ is a block with defect group $M:=\langle x^2,y,z\rangle$. Since $b_M^{D\C_G(M)}=b_D^{D\C_G(M)}$, also $b_M^{\C_G(a)}=b_a$ has defect group $M$. For every automorphism $\alpha\in\Aut(D)$ of order $3$ we have $\alpha(M)\ne M$.
Since $D$ controls the fusion of $B$-subpairs, we get $t(b_a)=l(b_a)=1$.

Now the conclusion follows from $k(B)=\sum_{a\in\mathcal{T}}{l(b_a)}$.
\end{proof}

The next result concerns the Cartan matrix of $B$.

\begin{Lemma}\label{cartanrs}
The elementary divisors of the Cartan matrix of $B$ are contained in $\{1,2,|D|\}$. The elementary divisor $2$ occurs twice and $|D|$ occurs once (as usual). In particular $l(B)\ge 3$.
\end{Lemma}
\begin{proof}
Let $C$ be the Cartan matrix of $B$. As in Lemma~\ref{eledivs} we use the notion of lower defect groups. For this let $P<D$ such that $|P|\ge 4$, and let $b\in\Bl(R\N_G(P))$ be a Brauer correspondent of $B$ with defect group $Q\le D$. Brauer's first main theorem implies $P<Q$. By Proposition~1.3 in \cite{Olsson} there exists a block $\beta\in\Bl(R\C_G(P))$ with $\beta^{\N_G(P)}=b$ such that at most $l(\beta)$ lower defect groups of $b$ contain a conjugate of $P$. Let $S\le Q$ be a defect group of $\beta$. First, we consider the case $S=D$. Then $P\subseteq\Z(D)$. By Lemma~\ref{uniquefix} we have $l(\beta)=1$, since $|P|\ge 4$. It follows that $m^1_b(P)=m_b(P)=0$, because $P$ is contained in the (lower) defect group $Q$ of $b$.

Now assume $S<D$. In particular $S$ is abelian. If $S$ is even metacyclic, then $l(\beta)=1$ and $m^1_b(P)=0$, since $P\subseteq\Z(\C_G(P))$. Thus, let us assume that $S$ is nonmetacyclic. By (3C) in \cite{BrauerBlSec1}, $x^2\in\Z(D)$ is conjugate to an element of $\Z(S)$. This shows $S\cong C_{2^k}\times C_{2^l}\times C_2$ with $k\in\{r,r-1\}$ and $1\le l\le r$. If $1,k,l$ are pairwise distinct, then $l(\beta)=1$ and $m^1_b(P)=0$ follow from Lemma~\ref{homocaut}. Let $k=l$. Then every automorphism of $S$ of order $3$ has only one nontrivial fixed-point. Since $|P|\ge 4$, it follows again that $l(\beta)=1$ and $m^1_b(P)=0$. 

Now let $S\cong C_{2^k}\times C_2^2$ with $2\le k\in\{r-1,r\}$. 
Assume first that $P$ is noncyclic. Then $S/P$ is metacyclic. If $S/P$ is not a product of two isomorphic cyclic groups, then $l(\beta)=1$ and $m^1_b(P)=0$. Hence, we may assume $S/P\cong C_2^2$. It is easy to see that there exists a subgroup $P_1\le P$ with $S/P_1\cong C_4\times C_2$. We get $l(\beta)=1$ and $m^1_b(P)=0$ also in this case.

Finally, let $P=\langle u\rangle$ be cyclic. Then $(u,\beta)$ is a $B$-subsection. Since $|P|\ge 4$, $u$ is not conjugate to $z$. As in the proof of Proposition~\ref{kminusl} we have $l(\beta)=1$ and $m^1_b(P)=0$. This shows $m^1_B(P)=0$. Since $P$ was arbitrary, the multiplicity of $|P|$ as an elementary divisor of $C$ is $0$.

It remains to consider the case $|P|=2$. We write $P=\langle u\rangle\le D$. As before let $b\in\Bl(R\N_G(P))$ be a Brauer correspondent of $B$. Then $(u,b)$ is a $B$-subsection. If $(u,b)$ is not conjugate to $(z,b_z)$, then $l(b)=1$ and $m^1_b(P)=0$ as in the proof of Proposition~\ref{kminusl}. Since we can replace $P$ by a conjugate, we may assume $P=\langle z\rangle$ and $(u,b)=(z,b_z)$. Then $l(b)=3$ and $D$ is a defect group of $b$. Now let $\overline{b}\in\Bl(R[\N_G(P)/P])$ be the block which is dominated by $b$. By Corollary~1 in \cite{DetCartan} the elementary divisors of the Cartan matrix of $\overline{b}$ are $1,1,|D|/2$. Hence, the elementary divisors of the Cartan matrix of $b$ are $2,2,|D|$. This shows
\[2=\sum_{\substack{Q\in\mathcal{P}(\N_G(P)),\\|Q|=2}}{m^1_b(Q)},\]
where $\mathcal{P}(\N_G(P))$ is a system of representatives for the conjugacy classes of $p$-subgroups of $\N_G(P)$. The same arguments applied to $b$ instead of $B$ imply $m^1_b(Q)=0$ for $P\ne Q\le\N_G(P)$ with $|Q|=2$. Hence, $2=m^1_b(P)=m^1_B(P)$, and $2$ occurs as elementary divisors of $C$ twice.
\end{proof}

As in Section~\ref{secrs1} we write $\IBr(b_u)=\{\phi_u\}$ for $u\in\mathcal{T}\setminus\langle z\rangle$. In a similar manner we define the integers $a^u_i$. If $u\in\mathcal{T}\setminus\langle z\rangle$ with $|\langle u\rangle|=2^k>2$, then the $2^{k-1}$ distinct subsections of the form $^\gamma(u,b_u)$ for $\gamma\in\mathcal{G}$ are pairwise nonconjugate (same argument as in the case $r>s=2$). Hence, Lemma~\ref{orthogonal} carries over in a corresponding form. Apart from that we can also carry over Lemma~(6.B) in \cite{Kuelshammerwr}:

\begin{Lemma}\label{heightzeroodd}
Let $\chi\in\Irr(B)$ and $u\in\mathcal{T}\setminus\Z(D)$. Then $\chi$ has height $0$ if and only if the sum \[\sum_{i=0}^{2^{r-1}-1}{a_i^u(\chi)}\] 
is odd.
\end{Lemma}
\begin{proof}
If $\chi$ has height $0$, the sum is odd by Proposition~1 in \cite{BroueSanta}. The other implication follows easily from (5G) in \cite{BrauerBlSec2}.
\end{proof}

The next lemma is the analogon to Lemma~\ref{neukirchapp222}.

\begin{Lemma}\label{neukirchapp}
Let $u\in\Z(D)\setminus\langle z\rangle$ of order $2^k$. Then for all $\chi\in\Irr(B)$ we have:
\begin{enumerate}[(i)]
\item $2^{h(\chi)}\mid a_i^u(\chi)$ for $i=0,\ldots,2^{k-1}-1$,\label{neukirchapp1}
\item $\sum\limits_{i=0}^{2^{k-1}-1}{a_i^u(\chi)}\equiv 2^{h(\chi)}\pmod{2^{h(\chi)+1}}$.\label{neukirchapp2}
\end{enumerate}
\end{Lemma}

As in the case $r>s=1$, Lemma~1.1 in \cite{Robinson} implies
\begin{equation}\label{kBleD}
k(B)\le\sum_{i=0}^{\infty}{2^{2i}k_i(B)}\le|D|.
\end{equation}
In particular Brauer's $k(B)$-conjecture holds. Moreover, Theorem~3.1 in \cite{Robinson} gives $k_0(B)\le|D|/2=|D:D'|$, i.\,e. Olsson's conjecture is satisfied. Using this, we can improve the inequality~\eqref{kBleD} to
\[|D|\ge k_0(B)+4(k(B)-k_0(B))=4k(B)-3k_0(B)\ge 4k(B)-\frac{3|D|}{2}\]
and
\[\frac{5\cdot 2^{2(r-1)}+16}{3}\le k(B)-l(B)+l(B)=k(B)\le\frac{5|D|}{8}=5\cdot 2^{2(r-1)}.\]

We will improve this further. Let $\overline{b_z}$ be the block of $\Bl(R\C_G(z)/\langle z\rangle)$ which is dominated by $b_z$. Then $\overline{b_z}$ has defect group $D/\langle z\rangle\cong C_{2^r}^2$. Using the existence of a perfect isometry (see \cite{Usami23I,Usami23II,UsamiZ4}), one can show that the Cartan matrix of $\overline{b_z}$ is equivalent to
\[\overline{C}:=\frac{1}{3}\begin{pmatrix}2^{2r}+2&2^{2r}-1&2^{2r}-1\\[1mm]2^{2r}-1&2^{2r}+2&2^{2r}-1\\[1mm]2^{2r}-1&2^{2r}-1&2^{2r}+2\end{pmatrix}.\]
Hence, the Cartan matrix of $b_z$ is equivalent to $2\overline{C}$. Now inequality $(\ast\ast)$ in \cite{KuelshammerWada} yields
\[k(B)\le 2\frac{2^{2r}+8}{3}=\frac{|D|+16}{3}.\]
(Notice that the proof of Theorem~A in \cite{KuelshammerWada} also works for $b_z$ instead of $B$, since the generalized decomposition numbers corresponding to $(z,b_z)$ are integral. See also Lemma~3 in \cite{SambalekB}.)

In addition we have 
\[k_i(B)=0\text{ for }i\ge 4\]
by Corollary~(6D) in \cite{BrauerApp4}.
This means that the heights of the characters in $\Irr(B)$ are bounded independently of $r$.
We remark also that Alperin's weight conjecture is equivalent to
\[l(B)=l(b)\] 
for the Brauer correspondent $b\in\Bl(R\N_G(D))$ of $B$ (see Consequence~5 in \cite{Alperinweights}). Since $z\in\Z(\N_G(D))$,  
$l(B)=l(b)=3$ and $k(B)=(5\cdot 2^{2(r-1)}+16)/3$ would follow in this case (see proof of Proposition~\ref{kminusl}).

\subsection{The gluing problem}
As in section~\ref{gluginsubsec} we use the notations of \cite{Parkgluing}.

\begin{Theorem}
 The gluing problem for $B$ has a unique solution.
\end{Theorem}
\begin{proof}
Let $\mathcal{F}$ be the fusion system induced by $B$. Then the $\mathcal{F}$-centric subgroups of $D$ are given by $M:=\langle x^2,y,z\rangle$ and $D$ (up to conjugation in $\mathcal{F}$). We have $\Aut_{\mathcal{F}}(M)\cong D/M\cong C_2$ and $\Aut_{\mathcal{F}}(D)\cong A_4$. This shows $\cohom^2(\Aut_{\mathcal{F}}(\sigma),F^\times)=0$ for every chain $\sigma$ of $\mathcal{F}$-centric subgroups. Consequently, $\cohom^0([S(\mathcal{F}^c)],\mathcal{A}^2_{\mathcal{F}})=0$. On the other hand, we have $\cohom^1(\Aut_{\mathcal{F}}(D),F^\times)\cong\cohom^1(C_3,F^\times)\cong C_3$ and $\cohom^1(\Aut_{\mathcal{F}}(\sigma),F^\times)=0$ for all chains $\sigma\ne D$. Hence, the situation is as in Case~3 of the proof of Theorem~1.2 in \cite{Parkgluing}. However, the proof in \cite{Parkgluing} is pretty short. For the convenience of the reader, we give a more complete argument. 

Since $[S(\mathcal{F}^c)]$ is partially orderd by taking subchains, one can view $[S(\mathcal{F}^c)]$ as a category, where the morphisms are given by the pairs of ordered chains. In particular $[S(\mathcal{F}^c)]$ has exactly five morphisms. With the notations of \cite{Webb} the functor $\mathcal{A}_{\mathcal{F}}^1$ is a \emph{representation} of $[S(\mathcal{F}^c)]$ over $\mathbb{Z}$. Hence, we can view $\mathcal{A}_{\mathcal{F}}^1$ as a module $\mathcal{M}$ over the incidence algebra of $[S(\mathcal{F}^c)]$. More precisely, we have
\[\mathcal{M}:=\bigoplus_{a\in\Ob[S(\mathcal{F}^c)]}{\mathcal{A}_{\mathcal{F}}^1(a)}=\mathcal{A}_{\mathcal{F}}^1(D)\cong C_3.\]
Now we can determine $\cohom^1([S(\mathcal{F}^c)],\mathcal{A}_{\mathcal{F}}^1)$ using Lemma~6.2(2) in \cite{Webb}. For this let $d:\Hom[S(\mathcal{F}^c)]\to\mathcal{M}$ a derivation. Then we have $d(\alpha)=0$ for all $\alpha\in\Hom[S(\mathcal{F}^c)]$ with $\alpha\ne(D,D)=:\alpha_1$. However, \[d(\alpha_1)=d(\alpha_1\alpha_1)=(\mathcal{A}_{\mathcal{F}}^1(\alpha_1))(d(\alpha_1))+d(\alpha_1)=2d(\alpha_1)=0.\] 
Hence, $\cohom^1([S(\mathcal{F}^c)],\mathcal{A}^1_{\mathcal{F}})=0$.
\end{proof}

\subsection{Special cases}

Since the general methods do not suffice to compute the invariants of $B$, we restrict ourself to certain special situations.

\begin{Proposition}\label{O2G}
If $\pcore_2(G)\ne 1$, then
\[k(B)=\frac{5\cdot 2^{2(r-1)}+16}{3},\hspace{1cm}k_0(B)\ge\frac{2^{2r}+8}{3},\hspace{1cm}l(B)=3.\]
\end{Proposition}
\begin{proof}
Let $1\ne Q:=\pcore_2(G)$. Then $Q\subseteq D$. In the case $Q=D'$ we have $\C_G(z)=\N_G(Q)=G$ and $B=b_z$. Then the assertions on $k(B)$ and $l(B)$ are clear. Moreover, $b_z$ dominates a block $\overline{b_z}\in\Bl(R\C_G(z)/\langle z\rangle)$ with defect group $C_{2^r}^2$. By Theorem~2 in \cite{Sambale} we have 
\[k_0(B)\ge k_0(\overline{b_z})=k(\overline{b_z})=\frac{2^{2r}+8}{3}.\] 
Hence, we may assume $Q\ne D'$. With the same argument we may also assume $Q<D$. In particular $Q$ is abelian. We consider a $B$-subpair $(Q,b_Q)$.  
Then $D$ or $M$ is a defect group of $b_Q$ (see proof of Lemma~\ref{repsystemrs}).
If $D$ is a defect group of $b_Q$, then $D\subseteq\C_G(Q)$ and $Q\subseteq\Z(D)$. By Lemma~\ref{uniquefix} it follows that $b_Q$ is nilpotent.

Now let us assume that $M$ is a defect group of $b_Q$. Since $D$ controls the fusions of $B$-subpairs, we have $t(b_Q)=1$ (see Case~2 in the proof of Proposition~\ref{kminusl}). Hence, again $b_Q$ is nilpotent. Thus, in both cases $B$ is an extension of a nilpotent block of $\Bl(R\C_G(Q))$. In this situation the Külshammer-Puig theorem applies. In particular we can replace $B$ by a block with normal defect group (see \cite{exnilblocks}). Hence, $B=b_z$, and the claim follows as before.
\end{proof}

Since $\N_G(D)\subseteq\C_G(z)$, $B$ is a “centrally controlled block” (see \cite{KuelshammerOkuyama}). In \cite{KuelshammerOkuyama} it was shown that then an epimorphism $\Z(B)\to\Z(b_z)$ exists, where one has to regard $B$ (resp. $b_z$) as blocks of $FG$ (resp. $F\C_G(z)$). Moreover, we conjecture that the blocks $B$ and $b_z$ are Morita-equivalent. For the similar defect group $Q_8$ this holds in fact (see \cite{KessarLinckelmanntame}). In this context the work \cite{Cabanes} is also interesting. There is was shown that there is a perfect isometry between any two blocks with the same quaternion group as defect group and the same fusion of subpairs. 
Thus, it would be also possible that there is a perfect isometry between $B$ and $b_z$.

\begin{Proposition}\label{OGreduction}
In order to determine $k(B)$ (and thus also $l(B)$), we may assume that $\pcore_2(G)$ is trivial and $\pcore_{2'}(G)=\Z(G)=\F(G)$ is cyclic. Moreover, we can assume that $G$ is an extension of a solvable group by a quasisimple group. In particular $G$ has only one nonabelian composition factor.
\end{Proposition}
\begin{proof}
By Proposition~\ref{O2G} we may assume $\pcore_2(G)=1$. Now we consider $\pcore(G):=\pcore_{2'}(G)$. Using Clifford theory we may assume that $\pcore(G)$ is central and cyclic (see e.\,g. Theorem X.1.2 in \cite{Feit}). Since $\pcore_2(G)=1$, we get $\pcore(G)=\Z(G)$. Let $\E(G)$ be the normal subgroup of $G$ generated by the components. As usual, $B$ covers a block $b$ of $\E(G)$. By Fong-Reynolds we can assume that $b$ is stable in $G$. Then $d:=D\cap\E(G)$ is a defect group of $b$. By the Külshammer-Puig result we may assume that $b$ is nonnilpotent. In particular $d$ has rank at least $2$. Let $C_1,\ldots,C_n$ be the components of $G$. Then $\E(G)$ is the central product of $C_1,\ldots,C_n$. Since $[C_i,C_j]=1$ for $i\ne j$, $b$ covers exactly one block $\beta_i$ of $RC_i$ for $i=1,\ldots,n$. Then $b$ is dominated by the block $\beta_1\otimes\ldots\otimes\beta_n$ of $R[C_1\times\ldots\times C_n]$. Since $\Z(C_1)$ is abelian and subnormal in $G$, it must have odd order. Hence, we may identify $b$ with $\beta_1\otimes\ldots\otimes\beta_n$ (see Proposition~1.5 in \cite{duevel}). In particular $d=\delta_1\times\ldots\times\delta_n$, where $\delta_i:=d\cap C_i$ is a defect group of $\beta_i$ for $i=1,\ldots,n$. Assume that $\delta_1$ is cyclic. Then $\beta_1$ is nilpotent and isomorphic to $(R\delta_1)^{m\times m}$ for some $m\in\mathbb{N}$ by Puig. Let $\{C_1,\ldots,C_k\}$ be the orbit of $C_1$ under the conjugation action of $G$ ($k\le n$). Then $\beta_1\otimes\ldots\otimes\beta_k\cong(R\delta_1)^{m_1\times m_1}$ (for some $m_1\in\mathbb{N}$) is a block of $R[C_1\ldots C_k]$ with $l(\beta_1\otimes\ldots\otimes\beta_k)=1$. Lemma~\ref{lemmamna}\eqref{lemmamna4} implies $k\le 2$ or $k=3$ and $|\delta_1|=2$. In the first case Theorem~2 in \cite{Sambale} shows that $\beta_1\otimes\ldots\otimes\beta_k$ is nilpotent. This also holds in the second case by \cite{Landrock2}.
Since $C_1\ldots C_k\unlhd G$, $B$ is an extension of a nilpotent block. This shows that we can assume that the groups $\delta_i$ are noncyclic for $i=1,\ldots,n$. By Lemma~\ref{lemmamna}\eqref{lemmamna4}, $d$ has rank at most $3$. Hence, $n=1$ and $\E(G)=C_1$. 

That means in order to determine the invariants of the block $B$ we may assume that $G$ contains only one component. Let $\F(G)$ (resp. $\F^\ast(G)$) be the Fitting subgroup (resp. generalized Fitting subgroup) of $G$. Since $\F(G)=\Z(G)$, we have $\C_G(\E(G))=\C_G(\F^\ast(G))\le\F(G)$. Hence, $\C_G(\E(G))$ is nilpotent. On the other hand, the quotient $G/\C_G(\E(G))$ is isomorphic to a subgroup of the automorphism group of the quasisimple group $\E(G)$. Consider the canonical map $f:\Aut(\E(G))\to\Aut(\E(G)/\Z(\E(G)))$. Let $\alpha\in\ker f$. Then $\alpha(g)g^{-1}\in\Z(\E(G))$ for all $g\in\E(G)$. Hence, we get a map $\beta:\E(G)\to\Z(\E(G))$, $g\mapsto\alpha(g)g^{-1}$. Moreover, it is easy to see that $\beta$ is a homomorphism. Since $\E(G)$ is perfect, we get $\beta=1$ and thus $\alpha=1$. This shows $\Aut(\E(G))\le\Aut(\E(G)/\Z(\E(G)))$. By Schreier's conjecture (which can be proven using the classification) $\Aut(\E(G)/\Z(\E(G)))$ is an extension of the solvable group $\Out(\E(G)/\Z(\E(G)))$ by the simple group $\Inn(\E(G)/\Z(\E(G)))\cong\E(G)/\Z(\E(G))$. Taking these facts together, we see that $G$ has only one nonabelian composition factor. In particular $G$ is an extension of a solvable group by a quasisimple group. 
\end{proof}

Now we consider blocks of maximal defect, i.\,e. $D$ is a Sylow $2$-subgroup of $G$. These include principal blocks.

\begin{Proposition}
If $B$ has maximal defect, then $G$ is solvable. In particular Alperin's weight conjecture is satisfied, and we have
\begin{align*}
k(B)&=\frac{5\cdot 2^{2(r-1)}+16}{3},\\
k_0(B)&=\frac{2^{2r}+8}{3},\\
k_1(B)&=\frac{2^{2(r-1)}+8}{3},\\
l(B)&=3.
\end{align*}
\end{Proposition}
\begin{proof}
By Feit-Thompson we may assume $\pcore_{2'}(G)=1$ in order to show that $G$ is solvable. We apply the $\Z^*$-theorem. For this let $g\in G$ such that $^gz\in D$. Since all involutions of $D$ are central (in $D$), we get $^gz\in\Z(D)$. By Burnside's fusion theorem there exists $h\in\N_G(D)$ such that $^hz={^gz}$. (For principal blocks this would also follow from the fact that $D$ controls fusion.) Since $D'=\langle z\rangle$, we have $^gz=z$. Now the $\Z^*$-theorem implies $z\in\Z(G)$. Then $D/\langle z\rangle\cong C_{2^r}^2$ is a Sylow $2$-subgroup of $G/\langle z\rangle$. By Theorem~1 in \cite{BrauerApp2}, $G/\langle z\rangle$ is solvable. Hence, also $G$ is solvable. Since Alperin's weight conjecture holds for solvable groups, we obtain the numbers $k(B)$ and $l(B)$. 

It is also known that the Alperin-McKay-conjecture holds for solvable groups (see \cite{Okuyama}). Thus, in order to determine $k_0(B)$ we may assume $D\unlhd G$. Then we can apply the results of \cite{Kuelshammer}. For this let $L:=D\rtimes C_3$. Then $B\cong (RL)^{n\times n}$ for some $n\in\mathbb{N}$. Hence, $k_0(B)$ is just the number of irreducible characters of $L$ with odd degree. By Clifford, every irreducible character of $L$ is an extension or an induction of a character of $D$. Thus, it suffices to count the characters of $L$ which arise from linear characters of $D$. These linear characters of $D$ are just the inflations of $\Irr(D/D')$. They spilt into the trivial character and orbits of length $3$ under the action of $L$ by Brauer's permutation lemma. The three inflations of $\Irr(L/D)$ are the extensions of the trivial character of $D$. The other linear characters of $D$ remain irreducible after induction. Characters in the same orbit amount to the same character of $L$. This shows
\[k_0(B)=3+\frac{|D/D'|-1}{3}=\frac{2^{2r}+8}{3}.\]

By Theorem~1.4 in \cite{Moreto} we have $k_i(B)=0$ for $i\ge 2$. We conclude
\[k_1(B)=k(B)-k_0(B)=\frac{5\cdot 2^{2(r-1)}+16}{3}-\frac{2^{2r}+8}{3}=\frac{2^{2(r-1)}+8}{3}.\qedhere\]
\end{proof}

The last result implies that Brauer's height zero conjecture is also satisfied for blocks of maximal defect. Moreover, the Dade-conjecture holds for solvable groups (see \cite{Dadepsolv}).

Finally we consider the case $r=2$ (i.\,e. $|D|=32$) for arbitrary groups $G$.

\begin{Proposition}
If $r=2$, we have
\[k(B)=12,\hspace{1cm}k_0(B)=8,\hspace{1cm}k_1(B)=4,\hspace{1cm}l(B)=3.\]
There are two pairs of $2$-conjugate characters of height $0$. The remaining characters are $2$-rational.
Moreover, the Cartan matrix of $B$ is equivalent to
\[\begin{pmatrix}4&2&2\\2&4&2\\2&2&12\end{pmatrix}.\]
\end{Proposition}
\begin{proof}
The proof is somewhat lengthy and consists entirely of technical calculations. For this reason we will only outline the argumenation. Since $k_0(B)$ is divisible by $4$, inequality~\eqref{kBleD} implies $k_0(B)\ge 8$. Since there are exactly two pairs of $2$-conjugate $B$-subsections, Brauer's permutation lemma implies that we also have two pairs of $2$-conjugate characters. Hence, the column $a_1^y$ contains at most four nonvanishing entries. Since $(a_1^y,a_1^y)=8$, there are just two nonvanishing entries, both are $\pm2$. Now Lemma~\ref{heightzeroodd} implies $k_0(B)=8$. This shows $(k(B),k_1(B),l(B))\in\{(12,4,3),(14,6,5)\}$.

By way of contradiction, we assume $k(B)=14$. Then one can determine the numbers $d^u_{\chi\phi}$ for $u\ne 1$ with the help of the contributions. However, there are many possibilities. The ordinary decomposition matrix $Q$ can be computed as the orthogonal space of the other columns of the generalized decomposition matrix. Finally we obtain the Cartan matrix of $B$ as $C=Q^{\text{T}}Q$. In all cases is turns out that $C$ has the wrong determinant (see Lemma~\ref{cartanrs}). This shows $k(B)=12$, $k_1(B)=4$ and $l(B)=3$.

Again we can determine the numbers $d^u_{\chi\phi}$ for $u\ne 1$. This yields the heights of the $2$-conjugate characters. We also obtain some informations about the Cartan invariants in this way. We regard the Cartan matrix $C$ as a quadratic form. Using the tables \cite{Nebeeven,Nebeodd} we conclude that $C$ has the form given in the statement of the proposition.
\end{proof}

\section*{Acknowledgment}
The author thanks his advisor Burkhard Külshammer for his encouragement. Proposition~\ref{OGreduction} was his idea.


\begin{thebibliography}{10}

\bibitem{Alperinweights}
J.~L. Alperin, \textit{Weights for finite groups}, in The {A}rcata {C}onference
  on {R}epresentations of {F}inite {G}roups ({A}rcata, {C}alif., 1986),
  369--379, Amer. Math. Soc., Providence, RI, 1987.

\bibitem{Bender}
H. Bender, \textit{Transitive {G}ruppen gerader {O}rdnung, in denen jede
  {I}nvolution genau einen {P}unkt festläßt}, J. Algebra \textbf{17} (1971),
  527--554.

\bibitem{Brauerconnection}
R. Brauer, \textit{On the connection between the ordinary and the modular
  characters of groups of finite order}, Ann. of Math. (2) \textbf{42} (1941),
  926--935.

\bibitem{BrauerApp2}
R. Brauer, \textit{Some applications of the theory of blocks of characters of
  finite groups. {II}}, J. Algebra \textbf{1} (1964), 307--334.

\bibitem{BrauerBlSec1}
R. Brauer, \textit{On blocks and sections in finite groups. {I}}, Amer. J.
  Math. \textbf{89} (1967), 1115--1136.

\bibitem{BrauerBlSec2}
R. Brauer, \textit{On blocks and sections in finite groups. {II}}, Amer. J.
  Math. \textbf{90} (1968), 895--925.

\bibitem{BrauerApp4}
R. Brauer, \textit{Some applications of the theory of blocks of characters of
  finite groups. {IV}}, J. Algebra \textbf{17} (1971), 489--521.

\bibitem{Brauer}
R. Brauer, \textit{On {$2$}-blocks with dihedral defect groups}, in Symposia
  {M}athematica, {V}ol. {XIII} ({C}onvegno di {G}ruppi e loro
  {R}appresentazioni, {INDAM}, {R}ome, 1972), 367--393, Academic Press, London,
  1974.

\bibitem{BroueSanta}
M. Broué, \textit{On characters of height zero}, in The {S}anta {C}ruz
  {C}onference on {F}inite {G}roups ({U}niv. {C}alifornia, {S}anta {C}ruz,
  {C}alif., 1979), 393--396, Amer. Math. Soc., Providence, R.I., 1980.

\bibitem{Buell}
D.~A. Buell, \textit{Binary quadratic forms}, Springer-Verlag, New York, 1989.

\bibitem{Cabanes}
M. Cabanes and C. Picaronny, \textit{Types of blocks with dihedral or
  quaternion defect groups}, J. Fac. Sci. Univ. Tokyo Sect. IA Math.
  \textbf{39} (1992), 141--161.

\bibitem{Dadeconj}
E.~C. Dade, \textit{Counting characters in blocks. {I}}, Invent. Math.
  \textbf{109} (1992), 187--210.

\bibitem{duevel}
O. D{\"u}vel, \textit{On {D}onovan's conjecture}, J. Algebra \textbf{272}
  (2004), 1--26.

\bibitem{Erdmann}
K. Erdmann, \textit{Blocks of tame representation type and related algebras},
  Lecture Notes in Mathematics, Springer-Verlag, Berlin, 1990.

\bibitem{Feit}
W. Feit, \textit{The representation theory of finite groups}, North-Holland
  Mathematical Library, North-Holland Publishing Co., Amsterdam, 1982.

\bibitem{DetCartan}
M. Fujii, \textit{On determinants of {C}artan matrices of {$p$}-blocks}, Proc.
  Japan Acad. Ser. A Math. Sci. \textbf{56} (1980), 401--403.

\bibitem{Isaacs}
I.~M. Isaacs, \textit{Character theory of finite groups}, AMS Chelsea
  Publishing, Providence, RI, 2006.

\bibitem{KessarLinckelmanntame}
R. Kessar and M. Linckelmann, \textit{On perfect isometries for tame blocks},
  Bull. London Math. Soc. \textbf{34} (2002), 46--54.

\bibitem{Kurzweil}
H. Kurzweil and B. Stellmacher, \textit{The theory of finite groups},
  Universitext, Springer-Verlag, New York, 2004.

\bibitem{Kuelshammerwr}
B. Külshammer, \textit{On 2-blocks with wreathed defect groups}, J. Algebra
  \textbf{64} (1980), 529--555.

\bibitem{Kuelshammer}
B. Külshammer, \textit{Crossed products and blocks with normal defect groups},
  Comm. Algebra \textbf{13} (1985), 147--168.

\bibitem{KuelshammerOkuyama}
B. Külshammer and T. Okuyama, \textit{On centrally controlled blocks of finite
  groups}, unpublished.

\bibitem{exnilblocks}
B. Külshammer and L. Puig, \textit{Extensions of nilpotent blocks}, Invent.
  Math. \textbf{102} (1990), 17--71.

\bibitem{KuelshammerWada}
B. Külshammer and T. Wada, \textit{Some inequalities between invariants of
  blocks}, Arch. Math. (Basel) \textbf{79} (2002), 81--86.

\bibitem{Landrock2}
P. Landrock, \textit{On the number of irreducible characters in a {$2$}-block},
  J. Algebra \textbf{68} (1981), 426--442.

\bibitem{gluingprob}
M. Linckelmann, \textit{Fusion category algebras}, J. Algebra \textbf{277}
  (2004), 222--235.

\bibitem{Mazurov}
V.~D. Mazurov, \textit{Finite groups with metacyclic {S}ylow 2-subgroups},
  Sibirsk. Mat. \v Z. \textbf{8} (1967), 966--982.

\bibitem{Moreto}
A. Moretó and G. Navarro, \textit{Heights of characters in blocks of
  {$p$}-solvable groups}, Bull. London Math. Soc. \textbf{37} (2005), 373--380.

\bibitem{Nagao}
H. Nagao and Y. Tsushima, \textit{Representations of finite groups}, Academic
  Press Inc., Boston, MA, 1989.

\bibitem{Nebeodd}
G. Nebe and N. Sloane, \textit{The {B}randt-{I}ntrau-{S}chiemann table of even
  ternary quadratic forms},
  \url{http://www2.research.att.com/~njas/lattices/Brandt_2.html}.

\bibitem{Nebeeven}
G. Nebe and N. Sloane, \textit{The {B}randt-{I}ntrau-{S}chiemann table of odd
  ternary quadratic forms},
  \url{http://www2.research.att.com/~njas/lattices/Brandt_1.html}.

\bibitem{Okuyama}
T. Okuyama and M. Wajima, \textit{Irreducible characters of {$p$}-solvable
  groups}, Proc. Japan Acad. Ser. A Math. Sci. \textbf{55} (1979), 309--312.

\bibitem{Olsson}
J.~B. Olsson, \textit{On {$2$}-blocks with quaternion and quasidihedral defect
  groups}, J. Algebra \textbf{36} (1975), 212--241.

\bibitem{OlssonRedei}
J.~B. Olsson, \textit{On the subsections for certain {$2$}-blocks}, J. Algebra
  \textbf{46} (1977), 497--510.

\bibitem{OlssonLDG}
J.~B. Olsson, \textit{Lower defect groups}, Comm. Algebra \textbf{8} (1980),
  261--288.

\bibitem{Olssonsubpairs}
J.~B. Olsson, \textit{On subpairs and modular representation theory}, J.
  Algebra \textbf{76} (1982), 261--279.

\bibitem{Parkgluing}
S. Park, \textit{The gluing problem for some block fusion systems}, J. Algebra
  \textbf{323} (2010), 1690--1697.

\bibitem{UsamiZ4}
L. Puig and Y. Usami, \textit{Perfect isometries for blocks with abelian defect
  groups and cyclic inertial quotients of order {$4$}}, J. Algebra \textbf{172}
  (1995), 205--213.

\bibitem{Robinson}
G.~R. Robinson, \textit{On {B}rauer's {$k(B)$} problem}, J. Algebra
  \textbf{147} (1992), 450--455.

\bibitem{Dadepsolv}
G.~R. Robinson, \textit{Dade's projective conjecture for {$p$}-solvable
  groups}, J. Algebra \textbf{229} (2000), 234--248.

\bibitem{Redei}
L. Rédei, \textit{Das „schiefe {P}rodukt“ in der {G}ruppentheorie},
  Comment. Math. Helv. \textbf{20} (1947), 225--264.

\bibitem{SambalekB}
B. Sambale, \textit{Cartan matrices and Brauer's {$k(B)$}-conjecture}, Journal
  of Algebra (to appear),
  \url{http://www.sciencedirect.com/science/article/B6WH2-51H1NP7-1/2/591ade320b5d95adf75b60f4fadb8ada}.

\bibitem{Sambale}
B. Sambale, \textit{Fusion systems on metacyclic {$2$}-groups},
  \url{http://arxiv.org/abs/0908.0783}.

\bibitem{Usami23I}
Y. Usami, \textit{On {$p$}-blocks with abelian defect groups and inertial index
  {$2$} or {$3$}. {I}}, J. Algebra \textbf{119} (1988), 123--146.

\bibitem{Usami23II}
Y. Usami, \textit{On {$p$}-blocks with abelian defect groups and inertial index
  {$2$} or {$3$}. {II}}, J. Algebra \textbf{122} (1989), 98--105.

\bibitem{Watanabe1}
A. Watanabe, \textit{Notes on {$p$}-blocks of characters of finite groups}, J.
  Algebra \textbf{136} (1991), 109--116.

\bibitem{Webb}
P. Webb, \textit{An introduction to the representations and cohomology of
  categories}, in Group representation theory, 149--173, EPFL Press, Lausanne,
  2007.

\end{thebibliography}
\end{document}